\documentclass[11pt,reqno]{amsart}

\usepackage{amsmath}
\usepackage{amsfonts}
\usepackage{amssymb}
\usepackage{amsthm}
\usepackage{mathrsfs}
\usepackage{latexsym}

\usepackage{cite} 

\usepackage{esint}

\numberwithin{equation}{section}

\usepackage{graphicx,subfigure}

\usepackage{ifthen} 

\provideboolean{shownotes} 
\setboolean{shownotes}{true} 
\newcommand{\margnote}[1]{
\ifthenelse{\boolean{shownotes}}%
{\marginpar{\raggedright\tiny\texttt{#1}}}%
{}%
}

\newcommand{\hole}[1]{
\ifthenelse{\boolean{shownotes}}%
{\begin{center} \fbox{ \rule {.25cm}{0cm}
\rule[-.1cm]{0cm}{.4cm} \parbox{.85\textwidth}{\begin{center}
\texttt{#1}\end{center}} \rule {.25cm}{0cm}}\end{center}}
{}
}

\usepackage[colorlinks=true,urlcolor=blue,
citecolor=red,linkcolor=blue,linktocpage,pdfpagelabels,
bookmarksnumbered,bookmarksopen]{hyperref}

\usepackage{fullpage}

\theoremstyle{plain}

\newtheorem{lemma}{Lemma}[section]
\newtheorem{theorem}[lemma]{Theorem}

\newtheorem{corollary}[lemma]{Corollary}

\theoremstyle{definition}
\newtheorem{remark}[lemma]{Remark}

\theoremstyle{remark}

\newcommand{\Id}{\mathrm{Id}}

\newcommand{\R}{\mathbb{R}}
\newcommand{\C}{\mathbb{C}}
\newcommand{\Z}{\mathbb{Z}}
\newcommand{\N}{\mathbb{N}}
\newcommand{\bbS}{\mathbb{S}}

\newcommand{\cO}{{\mathcal{O}}}
\newcommand{\cT}{{\mathcal{T}}}

\newcommand{\tcL}{\widetilde{{\mathcal{L}}}}
\newcommand{\cL}{{\mathcal{L}}}

\newcommand{\cS}{{\mathcal{S}}}

\newcommand{\cA}{{\mathcal{A}}}

\newcommand{\cV}{{\mathcal{V}}}
\newcommand{\cM}{{\mathcal{M}}}

\newcommand{\vep}{\varepsilon}

\renewcommand{\Re}{\mathrm{Re}\,} 
\renewcommand{\Im}{\mathrm{Im}\,}

\newcommand{\ep}{\epsilon}

\newcommand{\ess}{\sigma_\mathrm{\tiny{ess}}}
\newcommand{\ptsp}{\sigma_\mathrm{\tiny{pt}}}

\newcommand{\Ldper}{L^2_\mathrm{\tiny{per}}}

\newcommand{\Htper}{H^3_\mathrm{\tiny{per}}}
\newcommand{\Hmper}{H^m_\mathrm{\tiny{per}}}
\newcommand{\Hsper}{H^s_\mathrm{\tiny{per}}}

\newcommand{\Hsmdper}{H^{s-2}_\mathrm{\tiny{per}}}
\newcommand{\Hper}{H_\mathrm{\tiny{per}}}
\newcommand{\Hrper}{H^r_\mathrm{\tiny{per}}}

\newcommand{\<}{\langle}
\renewcommand{\>}{\rangle}

\makeatletter
\@namedef{subjclassname@2020}{\textup{2020} Mathematics Subject Classification}
\makeatother

\begin{document}

\title[Instability of periodic waves for generalized KdV-Burgers equations]{Orbital instability of periodic waves for generalized Korteweg-de Vries-Burgers equations with a source}

\author[A. Naumkina]{Anna Naumkina}

\address{{\rm (A. Naumkina)} Departamento de Matem\'aticas y Mec\'anica\\Instituto de 
Investigaciones en Matem\'aticas Aplicadas y en Sistemas\\Universidad Nacional Aut\'onoma de 
M\'exico\\ Circuito Escolar s/n, Ciudad Universitaria, C.P. 04510\\Cd. de M\'{e}xico (Mexico)}

\email{naumkinaanna75@gmail.com}

\author[R. G. Plaza]{Ram\'on G. Plaza}

\address{{\rm (R. G. Plaza)} Departamento de Matem\'aticas y Mec\'anica\\Instituto de 
Investigaciones en Matem\'aticas Aplicadas y en Sistemas\\Universidad Nacional Aut\'onoma de 
M\'exico\\ Circuito Escolar s/n, Ciudad Universitaria, C.P. 04510\\Cd. de M\'{e}xico (Mexico)}

\email{plaza@aries.iimas.unam.mx}

\begin{abstract}
A family of generalized Korteweg-de Vries-Burgers equations in one space dimension with a nonlinear source is considered. The purpose of this contribution is twofold. On one hand, the local well-posedness of the Cauchy problem on periodic Sobolev spaces and the regularity of the data-solution map are established. On the other hand, it is proved that periodic traveling waves which are spectrally unstable are also orbitally (nonlinearly) unstable under the flow of the evolution equation in periodic Sobolev spaces with same period as the fundamental period of the wave. This orbital instability criterion hinges on the well-posedness of the Cauchy problem, on the smoothness of the data-solution map, as well as on an abstract result which provides sufficient conditions for the instability of equilibria under iterations of a nonlinear map in Banach spaces. Applications of the former criterion include the orbital instability of a family of small-amplitude periodic waves with finite fundamental period for the Korteweg-de Vries-Burgers-Fisher equation, which emerge from a local Hopf bifurcation around a critical value of the velocity. 
\end{abstract}

\keywords{Periodic waves, KdV-Burgers equation, orbital instability, local well-posedness, viscous-dispersive semigroup}

\subjclass[2020]{35Q53, 35B35, 37L50, 35C07, 35B10}

\maketitle

\setcounter{tocdepth}{1}



\section{Introduction}
\label{secintro}

Consider the well-known generalized Korteweg-de Vries-Burgers equation,
\begin{equation}
\label{gKdVB}
u_t + f(u)_x - u_{xx} + u_{xxx} = 0,
\end{equation}
where $x \in \R$ and $t > 0$ denote the space and time variables, respectively, $u = u(x,t) \in \R$ is a scalar unknown, and $f$ is a nonlinear function. Equations of this form arise in the description of long nonlinear wave propagation in shallow water models (cf. Benney \cite{Benn66}, Gardner \cite{Grd96} and Jones \cite{Jo72}), as well as in weakly dissipative nonlinear plasma physics (see, for example, Grad and Hu \cite{GrHu67} and Hu \cite{HuPN72}). Equations \eqref{gKdVB} constitute a family of models (parametrized by the nonlinear flux function $f$) that generalizes the standard (normalized) Korteweg de Vries-Burgers (KdVB) equation,
\begin{equation}
\label{KdVB}
u_t + \alpha uu_x -  u_{xx} +  u_{xxx} = 0,
\end{equation}
with constant $\alpha > 0$, which was first derived under the weak nonlinearity and long wavelength approximations by Su and Gardner \cite{SuGa69} in order to describe a wide class of Galilean invariant physical systems (see also Ott and Sudan \cite{OtSu70}). The KdVB equation \eqref{KdVB} belongs, of course, to the family \eqref{gKdVB} upon selecting the Burgers' flux function $f(u) = \tfrac{\alpha}{2} u^2$.

Describing traveling wave solutions to nonlinear evolution equations has been one of the fundamental problems in theoretical and experimental physics, and the KdVB equation is no exception. The literature on nonlinear waves for the KdVB equation is very comprehensive; for an abridged list of works the reader is referred to Jeffrey and Kakutani \cite{JeKa72}, Johnson \cite{Jo72}, Canosa and Gazdag \cite{CaGaz77}, Bona and Schonbek \cite{BoScho85}, Xiong \cite{Xng89}, Pego \emph{et al.} \cite{PeSW93}, Schonbek and Rajopadhye \cite{SchoRa95}, Gardner \cite{Grd96} and the many references cited therein. The aforementioned works pertain to the study of traveling fronts, viscous-dispersive shock profiles or waves of solitary type, but not to spatially periodic traveling waves. In fact, Feng and Knobel \cite{FeKn07} proved that there are no periodic traveling wave solutions to the KdVB equation: indeed, the authors transform the underlying ODE into a two-dimensional autonomous system, for which a Poincar\'e phase plane analysis precludes the existence of periodic orbits. There exist, however, periodic traveling waves for KdVB-type equations with a linear source (cf. Feng and Kawahara \cite{FeKa00}; Mancas and Adams \cite{MaAd19}), or with forcing terms (see, e.g., Hattam and Clarke \cite{HtCl15a}). The asymptotics of solutions to the KdVB equation with a linear source term has been recently studied by Naumkin and Villela-Aguilar \cite{NauVil22} and the normal form of the KdVB equation with generic source terms was calculated by Kashchenko \cite{Ksh16}. In a recent contribution, Folino \emph{et al.} \cite{FNP24} proved the existence and the spectral instability of a family of small-amplitude periodic waves for the KdVB equation with a source of logistic (or monostable) type.

Motivated by the considerations above, in this paper we study generalized Korteweg-de Vries-Burgers equations with a generic source term of the form,
\begin{equation}
\label{genKdVBF}
u_t + f(u)_x -u_{xx} + u_{xxx} = g(u),
\end{equation}
with sufficiently smooth nonlinear functions $f$ and $g$. Essentially, equations \eqref{genKdVBF} describe the dynamics of a scalar quantity $u$ in an infinite one dimensional domain which is subject to four simultaneous mechanisms. First, the density $u$ is (nonlinearly) transported with speed $f'(u)$; $u$ is also subject to dispersive effects, encoded in the term $u_{xxx}$;  the diffusion of $u$ is represented by the term $u_{xx}$ and, finally, the reaction or source term $g(u)$ may describe production/consumption (with per capita rate $g(u)/u$), chemical reactions or combustion, among other interactions. 

In this work, we are interested in establishing new results related to the dynamics of periodic traveling waves for models of the form  \eqref{genKdVBF}. In particular, we focus on their stability under small perturbations, which is a fundamental property for understanding the real-world dynamics of models of evolution type. The main goal of this paper is to show that, if a periodic wave is spectrally unstable (that is, if the formal linearized operator around the wave has spectra with positive real part when acting on an appropriate periodic space) then it is also nonlinearly (orbitally) unstable under the flow of the evolution equation \eqref{genKdVBF}. Hence, we establish an instability criterion that warrants the orbital instability of the manifold generated by any spectrally unstable periodic wave under the flow of  \eqref{genKdVBF} in periodic Sobolev spaces with \emph{same period} as the fundamental period of the wave. No particular assumptions are made on the nonlinear functions $f$ and $g$, except for their regularity. The main ingredients of our analysis are (i) the local well-posedness theory for \eqref{genKdVBF}; (ii) the regularity of the data-solution map; and, (iii) an (important) abstract result by Henry {\it{et al.}} \cite{HPW82}, which essentially determines the instability of a manifold of equilibria under iterations of a nonlinear map with unstable linearized spectrum. In order to apply the theorem by Henry \emph{et al.}, some essential elements are needed, such as a suitable well-posedness theory and the property that the data-solution map is of class $C^2$.

Hence, we start by establishing the local well-posedness for equations of the form \eqref{genKdVBF}. Albeit such result is quite standard and relies on directly applying Banach's fixed point theorem, for the convenience of the reader we present a detailed (yet concise) proof of local well-posedness of the Cauchy problem for equations of the form \eqref{genKdVBF} in periodic Sobolev spaces of distributions; we closely follow the analysis of Iorio and Iorio \cite{IoIo01} for nonlinear equations. The main reason to include this material is that several refined estimates in the course of proof are needed in order to prove the smoothness of the data-solution map, an important key element of the abstract result by Henry \emph{et al.} Moreover, up to our knowledge, the well-posedness of equations of the form \eqref{genKdVBF} in Sobolev spaces of periodic distributions has not been reported before in the literature (the only related work that we know of is the well-posedness analysis in periodic spaces for the KdVB equation, without source terms, by Molinet and Vento \cite{MoVe13}).

Once the well-posedness and the regularity of the data-solution map are at hand, we proceed to prove the orbital instability criterion by verifying the hypotheses of the abstract theorem by Henry \emph{et al.} \cite{HPW82}. For that purpose, we show that the linearized problem around the wave is globally well-posed (upon application of standard semigroup theory) and define a suitable nonlinear map with unstable spectrum. We finish our study by directly applying the instability criterion to a particular example. Indeed, in the aforementioned recent paper by Folino \emph{et al.} \cite{FNP24}, the authors applied bifurcation techniques in order to show that there exists a family of periodic waves for the Korteweg-de Vries-Burgers-Fisher equation (see equation \eqref{KdVBFn} below). The family emerges from a local (subcritical) Hopf bifurcation when the speed $c$ crosses a critical value. These waves have small-amplitude and finite period. Folino \emph{et al.} also proved, upon application of perturbation theory of linear operators, that these waves are spectrally unstable. Therefore, our orbital instability criterion extends the instability property of these waves to a nonlinear level.

In the remainder of the Introduction we make precise the concepts of spectral and orbital stability so that we can state the main results of the paper.


\subsection{Periodic waves and their stability}

A \emph{spatially periodic} traveling wave for equation \eqref{genKdVBF} is a solution which has the form
\begin{equation}
\label{tws}
u(x,t) = \varphi(x-ct),
\end{equation}
where the constant $c \in \R$ is the speed of the wave and the profile function, $\varphi = \varphi(\xi)$, $\xi \in \R$, is a sufficiently smooth periodic function of its argument with fundamental period $L > 0$. Upon substitution of \eqref{tws} into \eqref{genKdVBF} we obtain the following equation for the profile,
\begin{equation}
\label{profileq}
-c \varphi' + f'(\varphi)\varphi' + \varphi'''= \varphi'' + g(\varphi),
\end{equation}
where $' = d/d\xi$ denotes differentiation with respect to the Galilean variable of translation, $\xi = x-ct$. With a slight abuse of notation let us transform the space variable as $x \mapsto x - c t$ in order to recast \eqref{genKdVBF} into the equation
\begin{equation}
\label{genKdVBFg}
u_t = - u_{xxx} + u_{xx} + g(u) +cu_x - f(u)_x,
\end{equation}
for which now the periodic wave is a stationary solution, $u(x,t) = \varphi(x)$, in view of \eqref{profileq}. For solutions to \eqref{genKdVBFg} of the form $\varphi(x) + v(x,t)$, where $v$ denotes a nearby perturbation, the leading approximation is given by the linearization of this equation around $\varphi$, namely
\[
v_t = - v_{xxx} + v_{xx} + (c - f'(\varphi(x)))v_x + (g'(\varphi(x)) - f'(\varphi(x))_x) v.
\]
Specializing to perturbations of the form $v(x,t) = e^{\lambda t} u(x)$, where $\lambda \in \C$ and $u$ lies in an appropriate Banach space $X$, we arrive at the eigenvalue problem
\begin{equation}
\label{primeraev}
\lambda u = - u_{xxx} + u_{xx} + (c - f'(\varphi(x)))u_x + (g'(\varphi(x)) - f'(\varphi(x))_x) u,
\end{equation}
in which the complex growth rate appears as the eigenvalue. Intuitively, a necessary condition for the wave to be ``stable" is the absence of eigenvalues with $\Re \lambda > 0$, precluding exponentially growing models at the linear level. Motivated by the notion of spatially localized, finite energy perturbations in the Galilean coordinate frame in which the periodic wave is stationary, we consider $X = L^2(\R)$ and define the linearized operator around the wave as
\begin{equation}
\label{linop}
\left\{
\begin{aligned}
\tcL^c \, &: \, L^2(\R) \longrightarrow L^2(\R),\\
\tcL^c \, &: = \, -\partial_x^3 + \partial_x^2 + a_1(x) \partial_x + a_0(x) \Id, 
\end{aligned}
\right.
\end{equation}
with dense domain $D(\tcL^c) = H^2(\R)$, and where the coefficients,
\begin{equation}
\label{defas}
\begin{aligned}
a_1(x) &:= c - f'(\varphi(x)),\\
a_0(x) &:= g'(\varphi(x)) - f'(\varphi(x))_x,
\end{aligned}
\end{equation}
are bounded and periodic, satisfying $a_j(x + L) = a_j(x)$ for all $x \in \R$, $j = 0,1$. $\tcL^c$ is a densely defined, closed operator acting on $L^2(\R)$ with domain $D(\tcL^c) = H^2(\R)$. Hence, the eigenvalue problem \eqref{primeraev} is recast as $\tcL^c u = \lambda u$ for some $\lambda \in \C$ and $u \in D(\tcL^c) = H^2(\R)$. The definition of spectral stability is the absence of spectrum in the unstable complex half-plane with $\Re \lambda > 0$. More precisely, we say that a bounded periodic wave $\varphi$ is \emph{spectrally stable} as a solution to the evolution equation \eqref{genKdVBF} if the $L^2$-spectrum of the linearized operator around the wave defined in \eqref{linop} satisfies $\sigma(\cL)_{|L^2(\R)} \cap \{\lambda \in \C \, : \, \Re \lambda > 0\} = \varnothing$. Otherwise we say that it is \emph{spectrally unstable}.

It is well-known, however, that differential operators with periodic coefficients have no point spectrum in $L^2$ and its spectrum is purely essential (or continuous), $\sigma(\tcL^c)_{|L^2(\R)} = \ess(\tcL^c)_{|L^2(\R)}$ (see Lemma 3.3 in Jones \emph{et al.} \cite{JMMP14}, or Lemma 59, p. 1487, in Dunford and Schwartz \cite{DunSch2}). In such situations it is customary to describe the spectrum in terms of Floquet multipliers of the form $e^{i\theta} \in \bbS^1$, $\theta \in \R$ (mod $2\pi$) via a \emph{Bloch transform} (cf. Gardner \cite{Grd97} and Kapitula and Promislow \cite{KaPro13}). Indeed, the purely essential spectrum $\sigma(\tcL^c)_{|L^2(\R)}$ can be written as the union of partial point spectra,
\begin{equation}
\label{Floquetrep}
\sigma(\tcL^c)_{|L^2(\R)} =  \!\!\bigcup_{-\pi<\theta \leq \pi}\ptsp(\cL^c_\theta)_{|\Ldper([0,L])},
\end{equation}
where the one-parameter family of Bloch operators,
\begin{equation}
\label{Blochop}
\left\{
\begin{aligned}
\cL^c_\theta &:= - (\partial_x + i\theta/L)^3 + (\partial_x + i\theta/L)^2 + a_1(x) (\partial_x + i \theta/L) + a_0(x) \Id,\\
\cL^c_\theta &: \Ldper([0,L]) \to \Ldper([0,L]),
\end{aligned}
\right.
\end{equation}
with domain $D(\cL^c_\theta) = \Htper([0,L])$, is parametrized by the Floquet exponent (or \emph{Bloch parameter}) $\theta \in (-\pi,\pi]$, and act on the periodic Sobolev space with same period $L > 0$ as the period of the wave. Since the family has compactly embedded domains in $\Ldper = \Ldper([0,L])$ then their spectrum consists entirely of isolated eigenvalues, $\sigma(\cL^c_\theta)_{|\Ldper} = \ptsp(\cL^c_\theta)_{|\Ldper} $. Moreover, they depend continuously on the Bloch parameter $\theta$, which may be regarded as a local coordinate for the spectrum $\sigma(\tcL^c)_{|L^2(\R)}$ (see Proposition 3.7 in \cite{JMMP14}), meaning that $\lambda \in \sigma(\tcL^c)_{|L^2(\R)}$ if and only if $\lambda \in \ptsp(\cL^c_\theta)_{|\Ldper}$ for some $\theta \in (-\pi,\pi]$. The parametrization \eqref{Floquetrep} is called the \emph{Floquet characterization of the spectrum} (for details, see \'Alvarez and Plaza \cite{AlPl21}, Jones \emph{et al.} \cite{JMMP14}, Kapitula and Promislow \cite{KaPro13}, Gardner \cite{Grd97} and the many references therein). As a consequence of \eqref{Floquetrep} we conclude that the periodic wave $\varphi$ is $L^2$-spectrally unstable if and only if there exists $\theta_0 \in (-\pi, \pi]$ for which
\[
\ptsp(\cL_{\theta_0}^{{c}})_{|\Ldper([0,L])} \cap \{ \lambda \in \C \, : \, \Re \lambda > 0\} \neq \varnothing.
\]

\begin{remark}
\label{remBlochzero}
Notice that when the Bloch parameter is $\theta = 0$, the expression of the operator $\cL_0^c$ coincides with that of the linearized operator around the wave in \eqref{linop}, but now acting on a periodic space:
\begin{equation}
\label{linopBloch0}
\left\{
\begin{aligned}
\cL_0^c \, &: \, \Ldper([0,L]) \longrightarrow \Ldper([0,L]),\\
\cL_0^c \, &: = \, -\partial_x^3 + \partial_x^2 + a_1(x) \partial_x + a_0(x) \Id.
\end{aligned}
\right.
\end{equation}
\end{remark}

Suppose that the spectral (in)stability of a periodic wave is established. Then a natural question is whether this spectral information implies the nonlinear stability of the wave with respect to the flow of the equation \eqref{genKdVBF}. First, note that if the profile function $\varphi = \varphi(\cdot)$ is smooth enough then it belongs to the periodic space $\Htper([0,L])$. Hence, one can compare the motion $\varphi(x - ct)$, as a solution to \eqref{genKdVBF}, to a general class of motions $u = u(x,t)$ evolving  from initial conditions, $u(0) =\psi$, that are close in some sense to $\varphi$. The notion of orbital stability is, thus, the property that $u(\cdot,t)$ remains close to $\varphi(\cdot +\gamma)$, $\gamma=\gamma(t)$, for all times provided that $u(0)$ starts close to $\varphi(\cdot)$. We define the orbit generated by $\varphi$ as the set
\[
\cO_\varphi = \{ \varphi(\cdot + r) \, : \, r \in \R \} \subset \Htper([0,L]).
\]
We note that $\cO_\varphi$ represents a $C^1$-curve, $\Gamma=\Gamma(r)$, in $\Htper([0,L])$ determined by the parameter $r\in \mathbb R$, $\Gamma(r)=\zeta_r( \varphi)$, where $\zeta$ denotes the translation operator, $\zeta_\eta(u) = u(\cdot + \eta)$. Thus, the traveling  wave profile will be orbitally stable if its orbit $\Gamma$ is stable by the flow generated by the evolution equation. Consequently, we have the following definition associated to \eqref{genKdVBF} (cf. Angulo \cite{AngAMS09}): Let $X$, $Y$ be Banach spaces, with the continuous embedding $Y \hookrightarrow X$; let $\varphi \in X$ be a traveling wave solution to equation \eqref{genKdVBF}. We say $\varphi$ is \emph{orbitally stable} in $X$ by the flow of \eqref{genKdVBF} if for each $\vep > 0$ there exists $\delta = \delta(\vep) > 0$ such that if $\psi \in Y$ and 
\[
\inf_{r \in \R} \| \psi(\cdot) - \varphi(\cdot + r) \|_X < \delta,
\]
then the solution $u(x,t)$ of \eqref{genKdVBF} with initial condition $u(0) = \psi$ exists globally and satisfies
\[
\sup_{t>0} \inf_{r \in \R} \| u(\cdot, t) - \varphi(\cdot + r) \|_X < \vep.
\] 
Otherwise we say that $\varphi$ is \emph{orbitally unstable} in $X$. In the present context of periodic waves for equations of the form \eqref{genKdVBF}, we choose $X = Y = \Htper([0,L])$, where $L > 0$ is the fundamental period of the wave.

\subsection{Main results}

The first main result of this paper pertains to the local well-posedness of the Cauchy problem for the generalized KdVB equation with a source \eqref{genKdVBF} on periodic Sobolev spaces. It also establishes the regularity of the data-solution map.

\begin{theorem}
[local well-posedness]
\label{teolocale}
Assume $f \in C^2(\R)$, $g \in C^1(\R)$, $L > 0$ and $s > 5/2$. If $\phi \in \Hsper([0,L])$ then there exist some ${T} = {T}(\| \phi \|_s) > 0$ and a unique solution $u \in C([0,T];\Hsper([0,L])) \cap C^1((0,T];\Hper^{s-2}([0,L]))$ to the Cauchy problem for equation \eqref{genKdVBF} with initial datum $u(0)=\phi$. For each $T_0\in (0,T)$, the data-solution map,
\[
\phi\in \Hsper([0,L]) \mapsto u_\phi\in C([0,T_0];\Hsper([0,L])),
\]
is continuous. Moreover, if we further assume $f \in C^4(\R)$ and $g \in C^3(\R)$ then the data-solution map is of class $C^2$.
\end{theorem}

The second main theorem establishes a general criterion for orbital instability of a periodic wave under the flow of the nonlinear evolution equation \eqref{genKdVBF}, based on an unstable spectrum of the linearized operator when posed on periodic Sobolev spaces with same period as the fundamental period of the wave.  

\begin{theorem}[orbital instability criterion]
\label{mainthem}
Suppose that $f \in C^4(\R)$, $g \in C^3(\R)$. Let $u(x,t) = \varphi (x-ct)$ be a periodic traveling wave solution with speed $c \in \R$ to the generalized KdV-Burgers equation \eqref{genKdVBF}, where the profile function $\varphi = \varphi(\cdot)$ is of class $C^3$ and has fundamental period $L > 0$. Assume that the following \emph{spectral instability property} holds: the linearized operator around the wave, namely
\begin{equation}
\label{lindoperc}
\left\{
\begin{aligned}
\cL^c &: \Ldper([0,L]) \to \Ldper([0,L]),\\
D(\cL^c) &= \Hper^3([0,L]) \subset \Ldper([0,L]),\\
\cL^c u &= - u_{xxx} + u_{xx} + g'(\varphi(x)) u - (f'(\varphi(x))u)_x + cu_x, \quad u \in \Htper([0,L]),
\end{aligned}
\right.
\end{equation}
has an unstable eigenvalue, that is, there exists $\lambda \in \C$ with $\Re \lambda > 0$ and some eigenfunction $\Psi \in D(\cL^c) = \Htper([0,L])$ such that $\cL^c \Psi = \lambda \Psi$. Then the periodic traveling wave is orbitally unstable in $\Hper^3([0,L])$ under the flow of equation \eqref{genKdVBF}.
\end{theorem}

\subsection*{Plan of the paper}

This paper is structured as follows. Section \ref{seclwp} contains the well-posedness theory for equations \eqref{genKdVBF} in periodic Sobolev spaces. Particular attention is paid to the regularity of the data-solution map. Section \ref{seccriterion} is devoted to proving the orbital instability criterion for periodic wave solutions to \eqref{genKdVBF}. In Section \ref{secapplication} we apply the former instability criterion to the family of small amplitude periodic waves to the Korteweg-de Vries-Burgers-Fisher equation (see equation \eqref{KdVBFn} below).

\subsection*{On notation}
We denote the real and imaginary parts of a complex number $\lambda \in \C$ by $\Re\lambda$ and $\Im\lambda$, respectively, as well as complex conjugation by $\overline{\lambda}$.  Linear operators acting on infinite-dimensional spaces are indicated with calligraphic letters (e.g., $\cL$), except for the identity operator which is indicated by $\Id$. The domain of a linear operator, $\cL : X \to Y$, with $X$, $Y$ Banach spaces, is denoted as $D(\cL) \subseteq X$. For a closed linear operator with dense domain the usual definitions of resolvent and spectra apply (cf. Kato \cite{Kat80}). When computed with respect to the space $X$, the spectrum of $\cL$ is denoted as $\sigma(\cL)_{|X}$.

\section{Local well-posedness}
\label{seclwp}

In this section we establish the local well-posedness in $\Hsper([0,L])$
for any $s\geq 5/2$ of the generalized KdVB equation \eqref{genKdVBF}.  In addition, we show that the data-solution map is of class $C^2$ if we require further regularity on the nonlinear functions $f$ and $g$.

\subsection{Preliminaries}

The classical Lebesgue and Sobolev spaces of complex-valued functions on the real line will be denoted as $L^2(\R)$ and $H^m(\R)$, with $m \in \N$, endowed with the standard inner products and norms. For any period $L > 0$ let $\mathscr{P} = C^\infty_{\mathrm{\tiny{per}}}([0,L])$ be the space of smooth $L$-periodic test functions and $\mathscr{P}'$ the space of $L$-periodic distributions. For any $s \in \R$, we denote by $\Hsper([0,L]) \subset \mathscr{P}'$ the Sobolev space of $L$-periodic distributions such that
\[
\| u \|^2_s := L \sum_{k=-\infty}^{\infty} (1 + |k|^2)^s |\widehat{u}(k)|^2 < \infty,
\]
where the sequence of the coefficients
\[
\widehat{u}(k)=\frac{1}{L}\int_{0}^{L} e^{-\frac{2\pi i}{L}kx} u(x) \,  dx, \qquad k \in \Z,
\]
is the Fourier transform of $u$. According to custom we denote $H^0_\mathrm{\tiny{per}} = \Ldper$. If $s > k + \tfrac{1}{2}$, $k \in \N \cup \{0\}$, then there holds the continuous embedding, $\Hsper \hookrightarrow C^k_\mathrm{\tiny{per}}$, where $C^k_\mathrm{\tiny{per}}$ is the space of $L$-periodic functions with $k$ continuous derivatives. The translation operator in $\Hsper([0,L])$ will be denoted as $\zeta_\eta : \Hsper([0,L]) \to \Hsper([0,L])$, $\zeta_\eta(u) = u(\cdot + \eta)$ for any $\eta \in \R$. Translation is a smooth operator in $\Hsper([0,L])$. Moreover, if $s \geq 0$ then we have $\| \zeta_\eta(u) \|_s = \| u \|_s$ for all $u \in \Hsper$ and all $\eta \in \R$ (see, e.g., Iorio and Iorio \cite{IoIo01}).

When $s = m \in \N \cup \{ 0 \}$, Sobolev spaces of periodic distributions can be characterized as follows. If $s = 0$ then $H^0_\mathrm{\tiny{per}} = \Ldper$ coincides with the space of complex $L$-periodic functions in $L^2_\mathrm{\tiny{loc}}(\R)$ satisfying $u(x + L) = u(x)$, a.e. in $x$, with inner product and norm given by
\[
\< u, v \>_0 = \int_0^L u(x) \overline{v(x)} \, dx, \qquad \| u \|^2_0 = \< u, u \>_0.
\]
For any $m \in \N$, $\Hmper([0,L])$ coincides with the space of functions
\[
\Hmper([0,L]) = \{ u \in H^m([0,L]) \, : \, \partial_x^j u (0) = \partial_x^j u(L), \; j=0,1,\ldots, m-1\}.
\]
The product and norm are given by
\[
\< u,v \>_m = \sum_{j=0}^m \< \partial_x^j u, \partial_x^j v\>_0, \quad \|u\|_m^2 = \< u,u\>_m.
\]
In the sequel (and for the rest of the paper) we use the notation,%
\[
\Hsper := \Hsper([0,L]) \; \text{ for any } s \in\mathbb{R},
\]
whenever there is no ambiguity in the choice of the period $L$. We start by studying the semigroup associated to the linear part of equation \eqref{genKdVBF}.

\subsection{The viscous-dispersive semigroup}

Consider the Cauchy problem for the following linear evolution equation in $\Hsper$,
\begin{equation}
 \label{LKdVB}
 \left\{
\begin{aligned}
u_t &= u_{xx} - u_{xxx}, \\
u(0) &= \phi,
\end{aligned}
\right.
\end{equation}
where $\phi \in \Hsper$ is the initial condition. Since we only consider the viscosity and dispersive terms of the linear part of equation \eqref{genKdVBF}, we denote the operator $\cT = -\partial_x^3 + \partial_x^2$ as the viscous-dispersive generator. This problem is a particular case of the problems studied by Iorio and Iorio \cite{IoIo01}, Section 4.2, of the form
\[
\left\{
\begin{aligned}
u &\in C([0,\infty); \Hsper),\\
u_t + i q(\partial) u &= \mu u_{xx},\\
u(0) &= \phi \in \Hsper,
\end{aligned}
\right.
\]
with $s \in \R$ fixed, $\mu \geq 0$ and $q = q(k) \in \mathscr{S}'$, $k \in \Z$, real valued so that
\[
q(\partial) f = \big( q(\cdot) \widehat{f}(\cdot)\big)^\vee, \qquad \text{ for all } \, f \in \mathscr{P}'.
\]
In our case, $q(k) = -k^3$, $\mu = 1$ and $i q(\partial) u =- i (k^3 \widehat{u})^\vee = u_{xxx}$.
%
\begin{theorem}
\label{localexistlinear}
The Cauchy problem \eqref{LKdVB} is globally well-posed in $\Hsper$ for any $s \in \R$. That is, for every $\phi \in \Hsper$ there exists a unique mild solution $u \in C([0,T];\Hsper)$ for all $T > 0$. The solution is given by
\[
u(t) = \cV(t) \phi,
\]
where the family of operators
\begin{equation}
\label{vdsg}
\left\{
\begin{aligned}
\cV(t) &: \Hsper \to \Hsper,\\
\cV(t) \phi &:= \sum_{k = - \infty}^\infty e^{-Q(k) t} \phi, \qquad t > 0,
\end{aligned}
\right.
\end{equation}
with
\begin{equation}
\label{Qks}
Q(k) := \Big(  \frac{2\pi}{L}\Big)^{2}k^{2}-i \Big(  \frac{2\pi} {L}\Big)^{3}k^{3}, \qquad k \in \Z,
\end{equation}
is a (viscous-dispersive) $C_0$-semigroup with infinitesimal generator $\cT = -\partial_x^3 + \partial_x^2$, $D(\cT) = \Hsper$. The solution depends continuously on the initial data in the following sense
\[
\sup_{t \in [0,\infty)} \| \cV(t) \phi_1 - \cV(t) \phi_2 \|_s \leq \| \phi_1 - \phi_2 \|_s,
\]
for all $\phi_j \in \Hsper$, $j = 1,2$.
\end{theorem}
\begin{proof}
This result is a particular case with $\mu = 1$ and $q(k) = -k^3$, $k \in \Z$, of the theory developed by Iorio and Iorio \cite{IoIo01}, section 4.2, pp. 218 - 223 (see, for instance, Theorems 4.9 and 4.14 in that reference). The expression for the semigroup \eqref{vdsg} - \eqref{Qks} can be obtained by a direct calculation, recalling that the fundamental period is $L> 0$.
\end{proof}

\begin{corollary}
\label{cordervdsg}
For all $s \in \R$, $N \geq 0$ and for every $\phi \in \Hsper$ there holds
\begin{equation}
\label{limdervdsg}
\lim_{h \to 0} \left\| h^{-1} \big( \cV(t+h) - \cV(t) \big) \phi - \partial_x^2 \big(  \cV(t) \phi \big) + \partial_x^3 \big(  \cV(t) \phi \big) \right\|_{s-N} = 0,
\end{equation}
uniformly with respect to $t \geq 0$. Moreover, there exists a uniform constant $\overline{C} > 0$ such that
\begin{equation}
\label{unifsgbd}
\left\| h^{-1} \big( \cV(h) - \Id) \big) - \partial_x^2 + \partial_x^3 \right\| \leq \overline{C},
\end{equation}
in the operator topology for all $ 0 \leq |h| \ll 1$, sufficiently small.
\end{corollary}
\begin{proof}
Follows from Theorem 4.15 in Iorio and Iorio \cite{IoIo01} upon definition of $\widetilde{Q}(\partial)f = (\partial_x^2 - \partial_x^3) f = ( -k^2 + ik^3) \widehat{f})^\vee$, $k \in \Z$. The second assertion follows immediately from \eqref{limdervdsg}.
\end{proof}

\begin{corollary}[smoothing estimate]
\label{corsmoothest}
For all $\phi \in \Hrper$, $r \in \R$ and all $\delta \geq 0$ there exists a uniform constant $K_\delta > 0$, depending only on $\delta$, such that
\begin{equation}
\label{smoothest}
\left\| \cV(t) \phi \right\|_{r + \delta} \leq K_\delta \Big[ 1 + \Big( \frac{1}{2t}\Big)^\delta \Big]^{1/2} \| \phi \|_{r},
\end{equation}
for all $t > 0$.
\end{corollary}
\begin{proof}
Follows from Theorem 4.17 in \cite{IoIo01} with $\mu = 1$ and $Q(\partial) = -\partial_x^3 + \partial_x^2$.
\end{proof}

Finally, we invoke a continuity result for the semigroup under consideration.

\begin{corollary}
\label{corcontsg}
For all $r \in \R$, $\delta \geq 0$ and $\epsilon > 0$, the mapping $t \mapsto \cV(t) \phi$ with $\phi  \in \Hrper$, is continuous with respect to the norm in $H^{r+\delta}_\mathrm{\tiny{per}}$. More precisely, for every $\alpha \in [0,1]$ there exists a constant $\widetilde{K} = \widetilde{K}(\delta,\alpha,\epsilon) > 0$ such that
\begin{equation}
\label{holdercont}
\| \cV(t) \phi - \cV(\tau) \phi \|_{r + \delta} \leq \widetilde{K} | t - \tau |^{\alpha} \| \phi \|_r,
\end{equation}
for all $t, \tau \geq \ep > 0$. In particular, $\cV(\cdot) \phi \in C((0, \infty); \mathscr{P})$.
\end{corollary}
\begin{proof}
See Theorem 4.18 in \cite{IoIo01}.
\end{proof}

\subsection{The Cauchy problem for generalized KdVB equations with a source}

Let us now consider the Cauchy problem in $\Hsper$, with $s > 5/2$, for generalized KdVB equations with a source of the form \eqref{genKdVBF}, with nonlinear functions $f \in C^2(\R)$ and $g \in C^1(\R)$. It can be written as
\begin{equation}
\label{CpKdVBF}
\begin{aligned}
u_t - u_{xx} + u_{xxx} &= F(u,u_x),\\
u(0) &= \phi,
\end{aligned}
\end{equation}
with initial condition $u(0) = \phi \in \Hsper$ and where we define $F \in C^2(\R^2)$ as
\[
F(u,p) := g(u) - f'(u)p, \qquad (u,p) \in \R^2.
\]
In order to solve the Cauchy problem we apply Duhamel's principle and define the mapping
\begin{equation}
\label{inteq}
\cA u := \cV(t) \phi + \int_0^t \cV(t-\tau) F(u, u_x)(\tau) \, d\tau.
\end{equation}

The local well-posedness of the Cauchy problem \eqref{CpKdVBF} will follow from a standard procedure (see, e.g., Taylor \cite{TayPDE3-2e}, chapter 15): first, the linear part of the equation generates a $C_0$ semigroup in a certain Banach space $X$ (this step has been already verified by Theorem \ref{localexistlinear}); second, the nonlinear term $F$ is locally Lipschitz from $X$ to another Banach space $Y$; and, finally, the operator $\cA$ is a contraction in a closed ball in $C([0,T];X)$ for $T$ sufficiently small yielding a solution to the integral equation \eqref{inteq}. We start by establishing the Lipschitz property for the function $F$.

\begin{lemma}
\label{lemFlip}
Let $s > 5/2$ and $\delta \in (\tfrac{3}{2},2)$. Assume that $f \in C^2$, $g \in C^1$. Then $F = F(u,u_x)$ is locally Lipschitz from $\Hsper$ to $\Hper^{s-\delta}$. More precisely, for any 
\[
u,v \in \overline{B_M} = \{ w \in \Hsper \, : \, \| w \|_s \leq M \} \subset \Hsper,
\]
with $M > 0$ fixed but arbitrary, there holds the estimate
\begin{equation}
\label{E1}
\| F(u,u_x) - F(v,v_x) \|_{s-\delta} \leq L_s(\| u \|_s, \| v \|_s) \| u - v \|_s,
\end{equation}
where $L_s : [0,\infty) \times [0,\infty) \to (0,\infty)$, $L_s = L_s(\zeta_1,\zeta_2)$, is a continuous, positive function and non-decreasing with respect to each argument. In particular, there holds the estimate
\begin{equation}
\label{E2}
\| F(u,u_x) \|_{s-\delta} \leq L_s(\| u \|_s, 0) \| u  \|_s,
\end{equation}
for all $u \in \overline{B_M}$.
\end{lemma}
\begin{proof}
Let $u,v \in \overline{B_M}$, $s > 5/2$ and $\delta \in (\tfrac{3}{2},2)$. Since for each $s-\delta > 1/2$, $\Hper^{s-\delta}$ is a Banach algebra (see Theorem 3.200 in \cite{IoIo01}), then there exists a constant $C_{s,\delta} > 0$, depending only on $s$ and $\delta$, such that
\[
\begin{aligned}
\| F(u,u_x) - F(v,v_x) \|_{s-\delta} &\leq \| g(u) - g(v) \|_{s-\delta} + \| f'(u) (v_x - u_x) \|_{s-\delta} \\
&\leq \| g(u) - g(v) \|_{s-\delta}+ C_{s,\delta} \| f'(u) \|_{s-\delta}\| u_x - v_x \|_{s-\delta}+ \\ &\quad + C_{s,\delta} \|f'(v) - f'(u) \|_{s-\delta}\| v_x \|_{s-\delta}.\\
\end{aligned}
\]
In view that $\| u \|_{s-\delta} \leq \| u \|_{s+1-\delta}$ and $\|u_x \|_{s-\delta} \leq \|u \|_{s+1-\delta}$ we obtain
\begin{equation}
\label{auxi1}
\begin{aligned}
\| F(u,u_x) - F(v,v_x) \|_{s-\delta} &\leq \| g(u) - g(v) \|_{s-\delta}  + C_{s,\delta} \| f'(u) \|_{s-\delta}\| u - v \|_{s+1-\delta}+ \\
&  \quad + C_{s,\delta} \|f'(v) - f'(u) \|_{s-\delta}\| v \|_{s+1-\delta}.\\
\end{aligned}
\end{equation}
On the other hand, it is well known that for any $s > \tfrac{1}{2}$ the $H^s$ norm controls the $L^\infty$ norm,
\[
\| u \|_\infty \leq \sum_{k \in \Z} |\widehat{u}(k)| \leq C_{s,L} \left( \sum_{k \in \Z} \Big( 1+\frac{4 \pi^2 k^2}{L^2}\Big)^s | \widehat{u}(k) |^2 \right)^{1/2} = C_{s,L} \| u \|_s,
\]
where
\[
C_{s,L} = \left( \sum_{k \in \Z} \Big( 1+\frac{4 \pi^2 k^2}{L^2}\Big)^{-s} \right)^{1/2} < \infty,
\]
is a uniform constant depending only on $s$ and $L$. Therefore, for all $u,v \in \overline{B_M}$ there holds $\|u\|_\infty, \|v\|_\infty \leq 2C_{s,L}M < \infty$ . Since $f, g$ and $f'$ are continuous in the compact set $[-2C_{s,L}M, 2C_{s,L}M]$ then they are locally Lipschitz. Therefore, we deduce the existence of uniform constants $L_f, L_g > 0$, depending only on $s, M$ and $\delta$, such that
\begin{equation}
\label{Lipest}
\begin{aligned}
\| f'(u) - f'(v) \|_{s-\delta} &\leq L_f \|u - v \|_{s-\delta},\\
\| g(u) - g(v) \|_{s-\delta} &\leq L_g \|u - v \|_{s-\delta},\\
\| g'(u) - g'(v) \|_{s-\delta} &\leq L_g \|u - v \|_{s-\delta},\\
\| f'(u) \|_{s-\delta} &\leq L_f \| u \|_{s-\delta} + |f'(0)|,\\
\| g'(u) \|_{s-\delta} &\leq L_g \| u \|_{s-\delta} + |g'(0)|,
\end{aligned}
\end{equation}
for all $u,v \in \overline{B_M}$. Upon substitution in estimate \eqref{auxi1} we arrive at
\[
\begin{aligned}
\| F(u,u_x) - F(v,v_x) \|_{s-\delta} &\leq L_g \| u -v \|_s + L_f C_{s,\delta} \| u \|_{s-\delta} \| u - v \|_{s+1-\delta} \\
&\quad  + C_{s,\delta} |f'(0)| \|u-v \|_{s+1-\delta}\\
&\quad + L_f C_{s,\delta} \|v\|_{s+1-\delta} \|u-v \|_{s-\delta}\\
&\leq \Big[ L_g + L_f C_{s,\delta} \| u \|_{s} + C_{s,\delta} |f'(0)| + L_f C_{s,\delta} \|v\|_{s}\Big] \| u- v \|_s,
\end{aligned}
\]
in view that $\delta > \tfrac{3}{2}$ implies that $\| u \|_{s+1-\delta} \leq \|u\|_s$ for all $u \in \Hsper$. Hence, we have proved estimate \eqref{E1} provided that we define,
\begin{equation}
\label{defiLipconst}
L_s(\zeta_1, \zeta_2) := L_f C_{s,\delta} \zeta_1 + L_f C_{s,\delta} \zeta_2 + L_g + C_{s,\delta} |f'(0)|  + C_{s,\delta} |g'(0)|,
\end{equation}
which clearly satisfies the desired properties. This proves the lemma.
\end{proof}

\begin{remark}
Notice that we have added the term $|g'(0)|$ to the expression for $L_s(\cdot, \cdot)$, a term which did not appear in the previous estimate. This will be useful later on.
\end{remark}

Now let $\phi \in \Hsper$, $s > 5/2$. For any $\beta > 0$ fixed but arbitrary, and for some $T > 0$ to be chosen later, we denote
\[
Z_{\beta,T} := \Big\{ u \in C([0,T]; \Hsper) \, : \, \sup_{t \in [0,T]} \| u(t) - \phi \|_s \leq \beta \Big\}.
\]
It is not hard to verify that $Z_{\beta,T}$ is closed under the norm
\[
\| u \|_{s,T} := \sup_{t \in [0,T]} \| u(t) \|_s.
\]
(details are left to the reader). We now prove that the operator defined in \eqref{inteq} is contractive on $Z_{\beta,T}$ for $T > 0$ chosen sufficiently small.

\begin{lemma}
\label{lemmildsol}
Let $s > 5/2$ and $\phi \in \Hsper$. Then we can find $T > 0$ sufficiently small such that the Cauchy problem \eqref{CpKdVBF} has a unique mild solution $u \in C([0,T];\Hsper)$ (that is,  a solution to the integral equation \eqref{inteq}). Moreover, the data-solution map, $\phi \mapsto u$, is continuous.
\end{lemma}
\begin{proof}
Let us show that if $u \in Z_{\beta, T}$ then $\cA u \in C([0,T];\Hsper)$. To that end, take $0 < t_1, < t_2 < T$. From the expression for $\cA$ we obtain,
\begin{equation}
\label{difA}
\begin{aligned}
\| \cA u(t_1) - \cA u(t_2) \|_s &\leq \| (\cV(t_1) - \cV(t_2) ) \phi \|_s \\
&+ \int_0^{t_1} \| (\cV(t_1-\tau) - \cV(t_2-\tau)) F(u,u_x)(\tau) \|_s \, d\tau\\
&+ \int_{t_1}^{t_2} \| \cV(t_2-\tau) F(u,u_x)(\tau) \|_s \, d\tau\\
&=: I_1 + I_2 + I_3.
\end{aligned}
\end{equation}

From Theorem \ref{localexistlinear}, $\{ \cV(t) \}_{t\geq 0}$ is a $C_0$-semigroup in $\Hsper$ and from standard semigroup properties we have that
\[
I_1 = \| (\cV(t_1) - \cV(t_2) ) \phi \|_s \to 0,
\]
as $t_2 \to t_1$ for any $\phi \in \Hsper$.

In order to compute the limit of $I_2$ we apply Theorem \ref{localexistlinear}, estimate \eqref{smoothest} with $\delta \in (\tfrac{3}{2}, 2)$ and Lemma \ref{lemFlip} (more precisely, estimate \eqref{E2} because $Z_{\beta,T} \subset \overline{B_M}$ with $M = \beta + \| \phi \|_s$ and with $r = s - \delta$). The result is
\[
\begin{aligned}
\| (\cV(t_1-\tau) - \cV(t_2-\tau)) &F(u,u_x)(\tau) \|_s \leq K_\delta \Big[ 1 + \Big( \frac{1}{2(t_1-\tau)} \Big)^{\delta} \Big]^{1/2} \| F(u,u_x)(\tau) \|_{s-\delta}\\
&\leq K_\delta \Big[ 1 + \Big( \frac{1}{2(t_1-\tau)} \Big)^{\delta} \Big]^{1/2} \!\!\sup_{\tau \in [0,T]} \big\{ L_s(\| u(\tau) \|_s,0) \| u(\tau) \|_s \big\},
\end{aligned}
\]
for all $0 < \tau < t_1$. Since $\delta < 2$, the right hand side of last inequality is integrable in $\tau \in (0,t_1)$. Hence, upon application of the Dominated Convergence Theorem, we obtain
\[
\lim_{t_2 \to t_1} I_2 =  \int_0^{t_1}  \lim_{t_2 \to t_1} \| (\cV(t_1-\tau) - \cV(t_2-\tau)) F(u,u_x)(\tau) \|_s \, d\tau = 0.
\]

In the same fashion, to compute the limit of $I_3$ we use estimates \eqref{smoothest} and \eqref{E2} to arrive at
\[
\begin{aligned}
\|\cV(t_2-\tau) F(u,u_x)(\tau) \|_s &\leq K_\delta \Big[ 1 + \Big( \frac{1}{2(t_2 - \tau)} \Big)^{\delta} \Big]^{1/2} \| F(u,u_x)(\tau) \|_{s-\delta}\\
&\leq K_\delta \Big[ 1 + \Big( \frac{1}{2(t_2 - \tau)} \Big)^{\delta} \Big]^{1/2}  L_s(\| u(\tau)\|_s, 0) \| u(\tau)\|_s,
\end{aligned}
\]
for every $\tau \in (t_1,t_2)$.  Now, if $u \in Z_{\beta,T}$ then we clearly have $\| u(\tau)\|_s \leq \beta + \| \phi \|_s$, yielding
\[
\begin{aligned}
I_3 &= \int_{t_1}^{t_2} \|\cV(t_2-\tau) F(u,u_x)(\tau) \|_s \, d\tau\\
&\leq K_\delta L_s(\beta + \| \phi \|_s,0) (\beta + \| \phi \|_s) \int_{t_1}^{t_2} \Big[ 1 + \Big( \frac{1}{2(t_2 - \tau)} \Big)^{\delta} \Big]^{1/2} \, d\tau\\
&\leq C K_\delta L_s(\beta + \| \phi \|_s,0) (\beta + \| \phi \|_s)  \big( t_2 - t_1 + (t_2 - t_1)^{1-\delta/2}\big) \to 0,
\end{aligned}
\]
as $t_2 \to t_1$, in view that $\delta \in (\tfrac{3}{2},2)$. Hence, we conclude that
\[
\| \cA u(t_1) - \cA u(t_2) \|_s \to 0, \qquad \text{as } \, t_2 \to t_1.
\]
This shows that $\cA u(t) \in \Hsper$ for all $t \in [0,T]$ and that $\cA u \in C([0,T]; \Hsper)$.

Now we shall choose $T > 0$ sufficiently small to guarantee that $\cA(Z_{\beta,T}) \subset Z_{\beta,T}$, as well as the contraction property for $\cA$. From elementary properties of $C_0$-semigroups we can always choose $T_1 > 0$ such that
\[
\| \cV(t) \phi - \phi \|_s < \tfrac{1}{2}\beta, \qquad \forall t \in [0,T_1].
\]
Take $u \in Z_{\beta,T}$. Hence, from estimate \eqref{E2} and for $\delta \in (\tfrac{3}{2},2)$ we have
\[
\begin{aligned}
\Big\| \int_{0}^{t} \cV(t -\tau) F(u,u_x)(\tau) \, d\tau \Big\|_s &\leq \int_{0}^{t} \big\| \cV(t-\tau) F(u,u_x)(\tau) \big\|_s \, d\tau  \\
&\leq K_\delta L_s(\beta + \| \phi \|_s,0) (\beta + \| \phi \|_s) \int_{0}^{t} \Big[ 1 + \Big( \frac{1}{2(t - \tau)} \Big)^{\delta} \Big]^{1/2} \, d\tau \\
&\leq C K_\delta L_s(\beta + \| \phi \|_s,0) (\beta + \| \phi \|_s) \int_{0}^{t} 1 + (t-\tau)^{-\delta/2} \, d\tau\\
&\leq C K_\delta L_s(\beta + \| \phi \|_s,0) (\beta + \| \phi \|_s) (T + T^{1-\delta/2})\\
&< \tfrac{1}{2}\beta,
\end{aligned}
\]
provided that we choose $T < T_1 \ll 1$ small enough. Hence, $\cA(Z_{\beta,T}) \subset Z_{\beta,T}$. 

Finally, let $u,v \in Z_{\beta,T}$. From similar arguments we obtain
\[
\begin{aligned}
\| \cA u(t) - \cA v(t) \|_s &\leq \int_0^{t} \| \cV(t-\tau) (F(u,u_x)(\tau) - F(v,v_x)(\tau)) \|_s \, d\tau\\
&\leq K_\delta \int_0^t \Big[ 1 + \Big( \frac{1}{2(t - \tau)} \Big)^{\delta} \Big]^{1/2} \| \cV(t-\tau) (F(u,u_x)(\tau) - F(v,v_x)(\tau)) \|_s \, d\tau\\
&\leq C K_\delta L_s(\beta + \| \phi \|_s, \beta + \| \phi \|_s) (T + T^{1-\delta/2}) \sup_{t \in [0,T]} \| u(t) - v(t) \|_s\\
&< \tfrac{1}{2} \| u - v \|_{s,T},
\end{aligned}
\]
where $0 < T \ll 1$ is sufficiently small such that
\begin{equation}
\label{boundC}
C_\phi (T + T^{1-\delta/2}) < \tfrac{1}{2},
\end{equation}
with $C_\phi = C K_\delta L_s(\beta + \| \phi \|_s, \beta + \| \phi \|_s) > 0$. Clearly, $T$ depends on $\| \phi \|_s$. Therefore, we conclude that there exists a sufficiently small time $T = T(\| \phi \|_s) > 0$ such that $\cA(Z_{\beta,T}) \subset Z_{\beta,T}$ and $\cA$ is a contractive mapping on $Z_{\beta,T}$. Upon application of Banach's fixed point theorem we deduce the existence of a unique fixed point $u \in Z_{\beta,T}$ of $\cA$, which is a solution to the integral equation \eqref{inteq} and, consequently, a unique mild solution to the Cauchy problem \eqref{CpKdVBF}.

Finally, in order to verify the continuity of the data-solution map, for given $\phi$ and $\psi$ in $\Hsper$, let us denote as $u$ and $v$ the solutions to the Cauchy problem with initial data $u(0) = \phi$ and $v(0) = \psi$, respectively. Then from Theorem \ref{localexistlinear}, Corollary \ref{corsmoothest} and Lemma \ref{lemFlip}, we obtain
\[
\begin{aligned}
\| u(t) - v(t) \|_s &\leq \| \cV(t) \phi - \cV(t) \psi \|_s + \int_0^t \| \cV(t-\tau)(F(u,u_x)(\tau) - F(v,v_x)(\tau)) \|_{s} \, d\tau \\
&\leq \| \phi - \psi \|_s + K_\delta \int_0^t \Big[ 1 + \Big( \frac{1}{2(t - \tau)} \Big)^{\delta} \Big]^{1/2} \| F(u,u_x)(\tau) - F(v,v_x)(\tau) \|_{s-\delta} \, d\tau\\
&\leq \| \phi - \psi \|_s + K_\delta L_s(M_s,M_s) (T + T^{1-\delta/2}) \int_0^t \| u(\tau) - v(\tau) \|_s \, d \tau,
\end{aligned}
\]
for any $\delta \in (\tfrac{3}{2},2)$  and where
\[
M_s := \max \Big\{ \sup_{t \in [0,T]} \| u(t) \|_s, \sup_{t \in [0,T]} \| v(t) \|_s \Big\}.
\]
Gronwall's inequality then yields
\[
\| u(t) - v(t) \|_s \leq C_{s,T} \| \phi - \psi \|_s,
\]
for all $t \in [0,T]$ and for some constant $C_{s,T} > 0$ depending only on $s$ and $T$. This shows the continuity of the mapping $\phi \mapsto u$ and the lemma is proved.
\end{proof}

Next, we prove that the unique mild solution from Lemma \ref{lemmildsol} is, in fact, a strong solution to the Cauchy problem.

\begin{lemma}
\label{lemstrongsol}
Under the hypotheses of Lemma \ref{lemmildsol}, the unique mild solution $u \in C([0,T]; \Hsper)$ to the Cauchy problem \eqref{CpKdVBF} satisfies $u \in C^1([0,T]; \Hsmdper)$.
\end{lemma}
\begin{proof}
We need to prove that
\begin{equation}
\label{limC1}
\lim_{h \to 0} \big\| h^{-1} \big(u(t+h) - u(t)\big) + u_{xxx} - u_{xx} - F(u,u_x) \big\|_{s-2} = 0.
\end{equation}
Substitute the expression for the mild solution of \eqref{inteq}. This yields,
\begin{equation}
\label{lae25}
\begin{aligned}
h^{-1} \big(u(t+h) - u(t)\big) + u_{xxx} - u_{xx} &- F(u,u_x) = h^{-1}  \big(\cV(t+h) - \cV(t) \big) \phi + \\
& + \partial_x^3 ( \cV(t) \phi )  - \partial_x^2 ( \cV(t) \phi ) +\\
&  + h^{-1} \int_0^t \big( \cV(t+h-\tau) - \cV(t-\tau) \big) F(u,u_x)(\tau) \, d\tau +\\
& + \partial_x^3 \int_0^t \cV(t-\tau) F(u,u_x)(\tau) \, d\tau +\\
& -  \partial_x^2 \int_0^t \cV(t-\tau) F(u,u_x)(\tau) \, d\tau +\\
& + h^{-1} \int_t^{t+h} \cV(t+h-\tau) F(u,u_x)(\tau) \, d\tau +\\
& - F(u,u_x)(t).
\end{aligned}
\end{equation}
It is to be noticed that Corollary \ref{cordervdsg} already implies that
\[
\big\| h^{-1} \big(\cV(t+h) - \cV(t) \big) \phi + \partial_x^3 ( \cV(t) \phi )  - \partial_x^2 ( \cV(t) \phi )  \big\|_{s-2} \to 0,
\]
as $h \to 0$. The last two terms of \eqref{lae25} are bounded by
\[
h^{-1} \int_t^{t+h} \cV(t+h-\tau) F(u,u_x)(\tau) \, d\tau - F(u,u_x)(t) \leq 
h^{-1}\int_t^{t+h} R(\tau) \, d \tau,
\]
where $R(\tau) := \| \cV(t+h-\tau) F(u,u_x)(\tau) - F(u, u_x)(t) \|_{s-2}$. Clearly, $R$ is continuous in $\tau \in (t, t + h)$ and, therefore, there exists some value $\eta \in (t, t + h)$ such that
\[
R(\eta) = h^{-1}\int_t^{t+h} R(\tau) \, d \tau.
\]
By continuity of the semigroup,
\[
\lim_{h \to 0} R(\eta) = \lim_{h \to 0} \| \cV(t+h-\eta) F(u,u_x)(\eta) - F(u, u_x)(t) \|_{s-2} = 0,
\]
because $\eta \to t$ as $h \to0$. Hence,
\[
0 \leq \lim_{h \to 0} \Big\| h^{-1} \int_t^{t+h} \cV(t+h-\tau) F(u,u_x)(\tau) \, d \tau - F(u, u_x)(t) \Big\|_{s-2} \leq \lim_{h \to 0} R(\eta) = 0.
\]

Now, from the smoothing estimate \eqref{smoothest} with $\delta \in (\tfrac{3}{2}, 2)$, from \eqref{unifsgbd} and \eqref{E2}, and from elementary properties of semigroups and their generators (the infinitesimal generator $\cT = \partial_x^3 - \partial^2_x$ commutes with the semigroup $\cV(\cdot)$), we obtain the following estimate for $0 < |h| \ll 1$ sufficiently small and for every $0 < \tau < t$:
\[
\begin{aligned}
\big\| h^{-1} (\cV(t+h &-\tau) - \cV(t-\tau)) F(u, u_x)(\tau) + \partial^3_x \big( \cV(t-\tau) F(u, u_x)(\tau)\big)- \partial^2_x \big( \cV(t-\tau) F(u, u_x)(\tau)\big) \big\|_{s-2} =\\
&\leq \| \cV(t-\tau)( h^{-1} (\cV(h) - \Id) + \partial_x^3 - \partial^2_x )F(u, u_x)(\tau) \|_{s-2}\\
&\leq K_\delta \Big[ 1 + \Big( \frac{1}{2(t - \tau)} \Big)^{\delta} \Big]^{1/2} \big\| (h^{-1} (\cV(h) - \Id) +\partial_x^3 - \partial^2_x ) F(u, u_x)(\tau) ) \big\|_{s-2-\delta}\\
&\leq \overline{C} K_\delta \Big[ 1 + \Big( \frac{1}{2(t - \tau)} \Big)^{\delta} \Big]^{1/2} \| F(u, u_x)(\tau) ) \|_{s-2-\delta}\\
&\leq \overline{C} K_\delta \Big[ 1 + \Big( \frac{1}{2(t - \tau)} \Big)^{\delta} \Big]^{1/2} \| F(u, u_x)(\tau) ) \|_{s-\delta}\\
&\leq \overline{C} K_\delta \Big[ 1 + \Big( \frac{1}{2(t - \tau)} \Big)^{\delta} \Big]^{1/2} L_s(\| u(\tau) \|_s,0) \| u(\tau)\|_s\\
&\leq \overline{C} K_\delta \Big[ 1 + \Big( \frac{1}{2(t - \tau)} \Big)^{\delta} \Big]^{1/2} L_s( \sup_{\tau \in(0,t)} \{ \| u(\tau) \|_s\},0) \sup_{\tau \in(0,t)} \{\| u(\tau)\|_s\}.
\end{aligned}
\]
Arguing as before, the right hand side of last inequality is integrable in $\tau \in (0,t)$ because $\delta \in (\tfrac{3}{2}, 2)$. Moreover, from Corollary \ref{cordervdsg} (see equation \eqref{limdervdsg} with $t = 0$) we have
\[
\lim_{h \to 0} \big\| (h^{-1}(\cV(h) - \Id) + \partial_x^3 - \partial^2_x )F(u, u_x)(\tau) \big\|_{s-2} = 0,
\]
uniformly with respect to $\tau \in (0,t)$. Therefore, from the Dominated Convergence Theorem we conclude that
\[
\begin{aligned}
0 \leq \lim_{h \to 0} &\left\| h^{-1} \int_0^t   (\cV(t+h -\tau) - \cV(t-\tau)) F(u, u_x)(\tau) \, d \tau  + \partial_x^3 \int_0^t  \cV(t-\tau) F(u, u_x)(\tau) \, d \tau + \right. \\
& \left. \qquad - \partial_x^2 \int_0^t  \cV(t-\tau) F(u, u_x)(\tau) \, d \tau  \right\|_{s-2}\\
&\leq \int_0^t \lim_{h \to 0} \left\| h^{-1} (\cV(t+h -\tau) - \cV(t-\tau)) F(u, u_x)(\tau) + \partial^3_x \big( \cV(t-\tau) F(u, u_x)(\tau)\big) + \right.\\
&\left. \qquad \qquad - \partial^2_x \big( \cV(t-\tau) F(u, u_x)(\tau)\big) \right\|_{s-2} \, d \tau \\
&\leq K_\delta \int_0^t \Big[ 1 + \Big( \frac{1}{2(t - \tau)} \Big)^{\delta} \Big]^{1/2}  \lim_{h \to 0}  \big\| (h^{-1} (\cV(h) - \Id) +\partial_x^3 - \partial^2_x ) F(u, u_x)(\tau) ) \big\|_{s-2} \, d \tau\\
&= 0.
\end{aligned}
\]
Combining all the limits above when $h \to 0$ we obtain \eqref{limC1}. The lemma is proved.
\end{proof}

\subsection{Regularity of the data-solution map}

In the proof of existence of the strong solution to the Cauchy problem \eqref{CpKdVBF} (see Lemmata \ref{lemFlip} through \ref{lemstrongsol}) we have assumed that $f \in C^2(\R)$ and $g \in C^1(\R)$. This implies that the data-solution map is continuous. In order to prove that the former is at least of class $C^2$ we need to impose further regularity on the nonlinear functions $f$ and $g$. Let us denote 
\[
B = B_\vep(\phi) := \{ u \in \Hsper \, : \, \| u - \phi \|_s < \vep \},
\]
for some $\vep > 0$ and define the mapping
\begin{equation}
\label{defGamma}
\left\{
\begin{aligned}
\Gamma &: B \times C([0,T];\Hsper) \to C([0,T];\Hsper),\\
\Gamma (\psi, w)(t) &:= w(t) - \cV(t) \psi - \int_0^t \cV(t-\tau) F(w,w_x)(\tau) \, d\tau.
\end{aligned}
\right.
\end{equation}

Now, for any $\phi \in \Hsper$, $s > 5/2$, let us denote by $u_\phi \in C([0,T]; \Hsper)$ the solution to
\begin{equation}
\label{uphi}
u_\phi(t) = \cV(t) \phi + \int_0^t \cV(t-\tau) F(u_\phi, \partial_x u_\phi) (\tau) \, d \tau,
\end{equation}
which is the unique solution to the Cauchy problem \eqref{CpKdVBF} with initial condition $u_\phi(0) = \phi$. This clearly implies that
\[
\Gamma(\phi, u_\phi) (t) = 0, \quad t \in [0,T],
\]
for all $\phi \in \Hsper$. 

\begin{lemma}
\label{lemFrechetdiff}
Let $f \in C^4(\R)$, $g \in C^3(\R)$ and $s > 5/2$. Then the map $\Gamma : \Hsper \times C([0,T];\Hsper) \to C([0,T];\Hsper)$ defined in \eqref{defGamma} is twice Fr\'echet differentiable in an open neighborhood $B_\vep(\phi) \times B_\eta(u_\phi)$ of $(\phi, u_\phi)$ for some $\vep > 0$ and $\eta >0$.
\end{lemma}
\begin{proof}
Because of the regularity of $F(u,p) = g(u) - f'(u) p \in C^3(\R^2)$ and the definition of the mapping $\Gamma$, it is easy to verify that there exist continuous G\^ateaux derivatives of up to second order in a sufficiently small open neighborhood of $(\phi, u_\phi)$. This implies, in turn, twice Fr\'echet differentiability of $\Gamma$ in a (perhaps smaller) neighborhood of $(\phi, u_\phi)$ (see, e.g.,  Zeidler \cite{ZeidI86}, \S 4.2, Proposition 4.8). This shows the result.
%
\end{proof}

\begin{lemma}
\label{leminvert}
Assume that $f \in C^4(\R)$, $g \in C^3(\R)$, $s > 5/2$ and $\phi\in \Hsper$. Consider the unique strong solution to the Cauchy problem \eqref{CpKdVBF}, $u_\phi \in C([0,T]; \Hsper) \cap C^1([0,T]; \Hsmdper)$, $T>0$, with initial condition $u_\phi(0) = \phi$, given by Lemma \ref{lemstrongsol}. Then, the operator
\begin{equation}
\label{derivwGamma}
\left\{
\begin{aligned}
\partial_w \Gamma (\phi, u_\phi) &: C([0,T];\Hsper) \to C([0,T];\Hsper),\\
\partial_w \Gamma (\phi, u_\phi) w(t) &= w(t) - \int_0^t \cV(t-\tau) \big( g'(u_\phi) w - (f'(u_\phi) w)_x  \big) \, d\tau,
\end{aligned}
\right.
\end{equation}
is one to one and onto. Moreover, the data-solution map,
\begin{equation}
\label{datasolmap}
\left\{
\begin{aligned}
\Upsilon &: \Hsper \to C([0,T];\Hsper),\\
\phi &\mapsto \Upsilon(\phi) = u_\phi,
\end{aligned}
\right.
\end{equation}
is of class $C^2$.
\end{lemma}
\begin{proof}
We need to compute
\[
\lim_{h \to 0} h^{-1} \Gamma(\phi, u_\phi + h w),
\]
for any $w \in C([0,T]; \Hsper)$. The resulting expression for the G\^ateaux derivative coincides with the Fr\'echet derivative, $\partial_w \Gamma(\phi, u_\phi)$, in an open neighborhood of $(\phi, u_\phi)$. From the expression for $F(u,u_x)$ we obtain
\[
F(u_\phi + hw, \partial_x u_\phi + hw_x) = F(u_\phi, \partial_x u_\phi) + h[ \big( g'(u_\phi) - f''(u_\phi) \partial_x u_\phi\big) w - f'(u_\phi) w_x \big)] + O(h^2).
\]
Therefore,
\[
h^{-1} \Gamma(\phi, u_\phi + hw)(t) = w(t) - \int_0^t \cV(t-\tau) \big( g'(u_\phi)w  - (f'(u_\phi)w)_x  \big) \, d\tau + O(h),
\]
in view of \eqref{uphi}. This yields the expression for the Fr\'echet derivative in \eqref{derivwGamma} when $h \to 0$.

Next, from the smoothing estimate \eqref{smoothest}, again with $\delta \in (\tfrac{3}{2},2)$, we arrive at
\[
\begin{aligned}
\| \partial_w \Gamma(\phi, u_\phi) w(t) - w(t) \|_s &\leq \int_0^t \| \cV(t-\tau) ( g'(u_\phi)w - (f'(u_\phi)w)_x ) \|_s \, d\tau\\
&\leq K_\delta \int_0^t \Big[ 1 + \Big( \frac{1}{2(t - \tau)} \Big)^{\delta} \Big]^{1/2} \| g'(u_\phi)w - (f'(u_\phi)w)_x  \|_{s-\delta} \, d\tau.
\end{aligned}
\]
We now use estimates \eqref{Lipest} and the fact that $\Hsper$ is a Banach algebra for every $s > \tfrac{1}{2}$ to estimate the integrand. Indeed, in a (perhaps smaller) neighborhood of $u_\phi$ such that $\| u_\phi \|_s \leq \beta + \| \phi \|_s$ (recall that $\cA$ satisfies $\cA(Z_{\beta,T}) \subset Z_{\beta,T}$), we obtain
\[
\begin{aligned}
\| g'(u_\phi)w - (f'(u_\phi)w)_x \|_{s-\delta} & \leq \| g'(u_\phi)w \|_{s-\delta} + \| (f'(u_\phi)w)_x \|_{s-\delta}\\
&\leq C_{s,\delta} \| g'(u_\phi) \|_{s-\delta} \| w \|_{s-\delta} + \| f'(u_\phi)w \|_{s+1-\delta}\\
&\leq C_{s,\delta} \big( L_g \| u_\phi \|_{s-\delta} + |g'(0)| \big) \| w \|_{s-\delta} + \| f'(u_\phi)w \|_{s}\\
&\leq C_{s,\delta} \big( L_g \| u_\phi \|_{s} + |g'(0)| + \| f'( u_\phi) \|_{s} \big) \| w \|_{s}\\
&\leq C_{s,\delta} \big( L_g \| u_\phi \|_{s} + |g'(0)| + L_f \| u_\phi \|_{s} + |f'(0)| \big) \| w \|_{s}\\
&\leq C L_s(\beta + \| \phi \|_s, 0) \sup_{\tau \in [0,T]} \| w(\tau) \|_s\\
&\leq C L_s(\beta + \| \phi \|_s, \beta + \| \phi \|_s) \sup_{\tau \in [0,T]} \| w(\tau) \|_s,
\end{aligned}
\]
for all $0 < \tau < t < T$. Upon integration we obtain,
\[
\begin{aligned}
\| \partial_w \Gamma(\phi, u_\phi) w(t) - w(t) \|_s &\leq C K_\delta L_s(\beta + \| \phi \|_s, \beta + \| \phi \|_s) (T + T^{1-\delta}) \sup_{\tau \in [0,T]} \| w(\tau) \|_s\\
&= C_\phi (T + T^{1-\delta}) \| w \|_{s,T},
\end{aligned}
\]
with the same constant $C_\phi = C K_\delta L_s(\beta + \| \phi \|_s, \beta + \| \phi \|_s)$ as in \eqref{boundC}. Hence, we conclude that
\[
\| (\partial_w \Gamma(\phi, u_\phi) - \Id) w \|_{s,T} < \tfrac{1}{2} \| w \|_{s,T},
\]
for all $w \in C([0,T]; \Hsper)$. This yields,
\[
\| \partial_w \Gamma(\phi, u_\phi) - \Id \| < \tfrac{1}{2} < 1,
\]
in the operator norm and, therefore, $\partial_w \Gamma(\phi, u_\phi)$ is invertible in $C[0,T]; \Hsper)$.

In view of Lemma \ref{lemFrechetdiff} we now apply the Implicit Function Theorem in Banach spaces (see Zeidler \cite{ZeidI86}, \S 4.7) to deduce the existence of an open neighborhood $\widetilde{B} \subset B$ of $\phi$ and a $C^2$-mapping
\[
\Upsilon : \widetilde{B} \to C([0,T];\Hsper),
\]
such that $\Gamma (w, \Upsilon(w)) = 0$ for all $w \in \widetilde{B}$. This mapping $\Upsilon$ is clearly the data-solution map because $\phi \mapsto \Upsilon(\phi) = u_\phi$. The lemma is proved.
\end{proof}

\subsection{Proof of Theorem \ref{teolocale}}
If we suppose that $f \in C^2(\R)$ and $g \in C^1(\R)$ then from lemmata \ref{lemmildsol} and \ref{lemstrongsol} we obtain the existence and uniqueness of a strong solution $u \in C([0,T]; \Hsper) \cap C^1([0,T]; \Hsmdper)$ to the Cauchy problem \eqref{CpKdVBF} with initial condition $u(0) = \phi \in \Hsper$, and the data-solution map, $\phi \mapsto u$, is continuous. If we further assume that $f \in C^4(\R)$ and $g \in C^3(\R)$ then Lemmata \ref{lemFrechetdiff} and \ref{leminvert} imply that the data-solution map is of class $C^2$. The theorem is proved. \qed

\section{Sufficient conditions for orbital instability}
\label{seccriterion}

This Section is devoted to prove Theorem \ref{mainthem}, which is a criterion for the orbital instability of periodic waves for a class of generalized KdV-Burgers equations with forcing of the form \eqref{genKdVBF}. We start by recalling an abstract result.

\subsection{The fundamental result by Henry, Perez and Wreszinski}

In 1982, Henry, Perez and Wreszinski \cite{HPW82} proved a general and abstract result that has provided the link to obtain nonlinear (orbital) instability from spectral instability of periodic waves in many situations; see, for example, the cases of viscous balance laws (\'Alvarez \emph{et al.} \cite{AAP24}), the KdV equation (Lopes \cite{LopO02}), the critical KdV and NLS models (Angulo and Natali \cite{AngNat09}) and even KdV systems (Angulo \emph{et al.} \cite{AnLN08}), just to mention a few. This theorem essentially determines the instability of a manifold of equilibria under iterations of a nonlinear map with unstable linearized spectrum. 

\begin{theorem}[Henry \emph{et al.} \cite{HPW82}]
\label{teohenry}
Let $Y$ be a Banach space and $\Omega \subset Y$ an open subset such that $0 \in \Omega$. Assume that there exists a map $\cM : \Omega \to Y$ such that $\cM(0) = 0$ and, for some $p > 1$ and some continuous linear operator $\cL$ with spectral radius $r(\cL) > 1$, there holds
\[
\| \cM(y) - \cL y \|_Y = O(\| y \|_Y^p) 
\]
as $y \to 0$. Then $0$ is unstable as a fixed point of $\cM$. More precisely, there exists $\vep_0 > 0$ such that for all $B_\eta(0) \subset Y$ and arbitrarily large $N_0 \in \N$ there is $n \geq N_0$ and $y \in B_\eta(0)$ such that $\| \cM^n(y) \|_Y \geq \vep_0$.
\end{theorem}
\begin{proof}
See Theorem 2 in Henry \emph{et al.} \cite{HPW82} (see also Theorem 5.1.5 in Henry \cite{He81}).
\qed \end{proof}

\begin{remark}
Theorem \ref{teohenry} guarantees the instability of the origin as a fixed point of $\cM$, inasmuch as it proves the existence of points moving away from the origin under successive applications of $\cM$. This result can be easily adapted to show the instability of a whole set of fixed points of $\cM$. Indeed, if $\Gamma_0$ is a $C^1$-curve of fixed points of  $\cM$ containing the origin then the whole curve $\Gamma_0$ is unstable, in other words, the points $\{\cM^n(y), n \geq 0\}$ not only move away from the origin but also from $\Gamma_0$ (see the Remark after Theorem 2 in Henry \emph{et al.}\cite{HPW82} for further details).
\end{remark}
%

Theorem \ref{teohenry} can be recast in a more suitable form for applications to nonlinear wave instability.

\begin{corollary}
\label{corhenry}
Let $\cS : \Omega \subset Y \to Y$ be a $C^2$ map defined on an open neighborhood of a fixed point $\varphi$ of $\cS$. If there is an element $\mu \in \sigma(\cS'(\varphi))$ with $|\mu| > 1$ then $\varphi$ is unstable as a fixed point of $\cS$. Moreover, if $ \Gamma$ is a $C^1$-curve of fixed points of  $\cS$ with $\varphi \in \Gamma$ then $\Gamma$ is unstable.
\end{corollary}
\begin{proof}
See the proof of Corollary 3 in \'Alvarez \emph{et al.} \cite{AAP24}.
\end{proof}

\subsection{Global existence of solutions to the linear problem}

Before the application of the abstract orbital stability result stated in Theorem \ref{teohenry}, first we need to verify that the Cauchy problem for the linearized equation around the periodic wave is globally well-posed in a Sobolev periodic space with same fundamental period as the wave.

\begin{theorem}[global well-posedness of the linearized problem]
\label{lemglobalwp}
Suppose that $f \in C^4(\R)$, $g \in C^3(\R)$ and let $\varphi = \varphi (x-ct)$ be a periodic traveling wave solution with speed $c \in \R$ to the generalized KdV-Burgers equation \eqref{genKdVBF}, where the profile function $\varphi = \varphi(\cdot)$ is of class $C^3$ and has fundamental period $L > 0$. Then for every $\phi \in \Htper([0,L])$ and all $T > 0$ there exists a unique solution $v_\phi \in C([0,T];\Htper([0,L])) \cap C^1([0,T]; \Ldper([0,L]))$ to the linear Cauchy problem 
\begin{equation}
\label{Cplin}
\left\{
\begin{aligned}
v_t &= \cL^c v, \\
v(0) &= \phi,
\end{aligned}
\right.
\end{equation}
where $\cL^c$ is the linearized operator around the periodic traveling wave $\varphi$,
\[
\begin{aligned}
\cL^c &: \Ldper([0,L]) \to \Ldper([0,L]) , \\
D(\cL^c) &= \Htper([0,L]) ,\\
\cL^c &= - \partial_x^3 + \partial_x^2 + a_1(x) \partial_x + a_0(x) \Id,
\end{aligned}
\]
and with real, $L$-periodic and at least of class $C^1$ uniformly bounded coefficients, given by
\[
\begin{aligned}
a_1(x) &= c - f'(\varphi(x)),\\
a_0(x) &= g'(\varphi(x)) - f''(\varphi(x)) \varphi'(x), & \qquad x \in [0,L].
\end{aligned}
\]
\end{theorem}
\begin{proof}
The operator $\cL^c$ is a closed, densely defined operator in $\Ldper([0,L])$. Since its domain is compactly embedded in $\Ldper([0,L])$ then its spectrum consists entirely of isolated eigenvalues, $\sigma(\cL^c)_{|\Ldper} = \ptsp(\cL^c)_{|\Ldper}$ (see, e.g., Kapitula and Promislow \cite{KaPro13}). Let us suppose that for $\lambda \in \C$, $u \in D(\cL^c) = \Htper([0,L])$ and $h \in \Ldper([0,L])$, there holds
\[
(\lambda - \cL^c) u = h.
\]
Since $s = 3 \in \N$, then $\Htper([0,L])$ can be identified as the space of complex $L$-periodic functions in $L^2_\mathrm{\tiny{loc}}(\R)$ with $\partial_x^j u(0) = \partial_x^j u(L)$, $j = 0,1,2$ with $\| u \|_3 < \infty$. Hence, take the complex $L^2$-product of $u$ with last equation to obtain,
\begin{equation}
\label{laeq1}
\lambda \| u \|_0^2 = \langle u, \cL^c u \rangle_0 + \langle u, h \rangle_0.
\end{equation}
Integrating by parts and because of the boundary conditions we have
\[
-\langle u, u_{xxx}\rangle_0 = \langle u_x, u_{xx} \rangle_0 = - \langle u_{xx}, u_x \rangle_0,
\]
yielding
\[
- \Re \langle u, u_{xxx}\rangle_0 = \tfrac{1}{2} \big( \langle u_x, u_{xx} \rangle_0 + \langle u_{xx}, u_{x} \rangle_0\big) = 0.
\]
Likewise,
\[
\langle u,u_{xx}\rangle_0 = - \langle u_x, u_x \rangle = - \| u_x \|_0^2.
\]
Moreover,
\[
\begin{aligned}
\Re \langle u, a_1(x) u_x\rangle_0 = \int_0^L a_1(x) \Re (u \, \overline{u_x} )\, dx &= \tfrac{1}{2} \int_0^L a_1(x) \partial_x |u|^2 \, dx\\
&= -\tfrac{1}{2} \int_0^L a_1'(x) |u|^2 \, dx + \left.\Big( \tfrac{1}{2} a_1(x) |u|^2 \Big)\right|_0^L\\
&= - \tfrac{1}{2} \langle u, a_1'(x) u \rangle_0,
\end{aligned}
\]
and
\[
\Re \langle u, a_0(x) u \rangle_0 = \int_0^L a_0(x) |u|^2 = \langle u, a_0(x) u \rangle_0,
\]
because the coefficients are real, periodic, differentiable and uniformly bounded. Taking the real part of \eqref{laeq1} and substituting we obtain
\[
(\Re \lambda) \| u \|_0^2 = \Re \langle u, \cL^c u \rangle_0 + \Re \langle u, h \rangle_0 = - \| u_x \|_0^2 + \Re \langle u, (a_0(x) - \tfrac{1}{2}a_1'(x)) u \rangle_0 + \Re \langle u, h \rangle_0.
\]
From the hypotheses on the coefficients there exists a uniform constant $C_0 > 0$, depending only on $c$ and $\varphi$, such that
\[
\sup_{x \in [0,L]} | a_0(x) - \tfrac{1}{2}a_1'(x) | \leq C_0.
\]
Therefore,
\[
(\Re \lambda) \| u \|_0^2 \leq \Re \langle u, (a_0(x) - \tfrac{1}{2}a_1'(x)) u \rangle_0 + \Re \langle u, h \rangle_0 \leq C_0 \| u \|_0^2 + \| u \|_0 \| h \|_0.
\]
If we take $\lambda \in \C$ such that $\Re \lambda > C_0$ then we obtain 
\[
\| u \|_0 \leq \frac{\| h \|_0}{\Re \lambda - C_0}.
\]
This shows that there exists $C_0 > 0$ such that, if $\Re \lambda > C_0$ then $\lambda \in \rho(\cL^c)$ and the following resolvent estimate holds,
\begin{equation}
\label{resolvest}
\| (\lambda - \cL^c)^{-1} \| \leq \frac{1}{\Re \lambda - C_0}.
\end{equation}

We now apply the classical Hille-Yosida theorem in the quasicontractive case (see, for instance, Engel and Nagel \cite{EN06}, Corollary 3.6, p. 68) to conclude that $\cL^c$ is the infinitesimal generator of a $C_0$-semigroup of quasicontractions, $\{ e^{t \cL^c} \}_{t \geq 0}$. Upon application of the standard theory of semigroups we conclude the existence of a unique global solution 
\[
v_\phi := e^{t \cL^c} \phi \in C([0,T];\Htper([0,L])) \cap C^1([0,T]; \Ldper([0,L])),
\] 
for all $T > 0$ to the linear Cauchy problem \eqref{Cplin} (see, e.g., Theorem 1.3 in Pazy \cite{Pazy83}, p. 102).
\end{proof}

\subsection{The mapping $\cS$}

Now we specify the particular mapping $\cS$ (in the context of Corollary \ref{corhenry}) suitable for our needs. First, it is to be observed that since $\varphi = \varphi(\cdot)$ is a $L$-periodic, $C^3$ profile function that defines the wave, then clearly $\varphi \in 
\Htper([0,L])$ as a function of the Galilean variable of translation, $x - ct$. Therefore, we can consider $\varphi = \varphi(x)$, $x \in [0,L]$, as the initial condition of the Cauchy problem. Thus, if $u_\varphi = \Upsilon(\varphi) \in C([0,T];\Htper([0,L]))$ denotes the unique solution to the Cauchy problem \eqref{CpKdVBF} with initial datum $\varphi \in \Htper([0,L])$, then for each $x \in [0,L]$ a.e. there holds $u_\varphi(t)(x) = \varphi(x-ct)$, or, equivalently, 
\begin{equation}
\label{solprofil}
u_\varphi(t) = \varphi(\cdot - ct) = \zeta_{-ct}(\varphi) \in \Htper([0,L]), 
\end{equation}
where $\zeta_\eta$ is the translation operator in $\Htper([0,L])$ for any $\eta \in \R$. This follows by direct differentiation and by the profile equation \eqref{profileq}. Therefore, for each $\phi \in \Htper([0,L]) $ let us define
\begin{equation}
\label{defS}
\begin{aligned}
\cS &: \Htper([0,L]) \to \Htper([0,L]),\\
\cS(\phi) &:= \zeta_{cT} (u_\phi(T))
\end{aligned}
\end{equation}
where $u_\phi = \Upsilon(\phi)$ denotes the unique solution to the Cauchy problem \eqref{CpKdVBF} with $u_\phi(0)=\phi$, $u_\phi \in C([0,{T}];\Htper([0,L]))$. Recall that $u_\phi$ is given by the variation of constants formula \eqref{uphi}.

\begin{lemma}
\label{propS} 
Let $\varphi$ be a periodic profile for equation \eqref{genKdVBF}. The mapping $\cS$ defined in \eqref{defS} satisfies:
\begin{itemize}
\item[(a)] $\cS(\varphi) = \varphi \in \Htper([0,L])$.
\item[(b)] $\cS$ is twice Fr\'echet differentiable in an open neighborhood of $\varphi$.
\item[(c)] For every $\psi \in \Htper([0,L])$ there holds
\begin{equation}
\label{devS}
\cS'(\varphi) \psi = v_\psi(T),
\end{equation}
where $v_\psi (t) \in \Htper([0,L])$ denotes the unique solution to the linear Cauchy problem \eqref{Cplin} with initial datum $v_\psi (0) = \psi$.
\end{itemize}
\end{lemma}
\begin{proof}
(a) Observing that $\cS(\varphi) = \zeta_{cT}(u_{\varphi}(T)) = \zeta_{cT} (\zeta_{-cT}(\varphi)) = \varphi$ in view of \eqref{solprofil}, we immediately reckon that $\varphi \in \Htper([0,L])$ is a fixed point of $\cS$.

(b) From Lemma \ref{leminvert} we know that the data-solution map $\phi \mapsto \Upsilon(\phi) = u_\phi$ is of class $C^2$. Also, the translation operator is of class $C^2$ in $\Htper([0,L])$ (actually, it is of class $C^\infty$). Thus, the composition is of class $C^2$ and we conclude that $\cS$ is twice Fr\'echet differentiable in an open neighborhood, $\Omega = \{ \phi \in \Htper([0,L]) \, : \, \| \phi - \varphi \|_2 < \eta \}$, of $\varphi$. 

(c) In order to prove (c), we ought to compute the Fr\'echet derivative of $\cS$ at $\varphi$. To that end, we calculate the (G\^ateaux derivative) operator,
\[
\cS'(\varphi) \psi = \frac{d}{d\vep} \big( \cS(\varphi + \vep \psi) \big)_{|\vep = 0},
\]
for any arbitrary $\psi \in \Htper([0,L])$. Notice that, by definition, $\cS(\varphi + \vep \psi) = \zeta_{cT} \big( u_{\varphi + \vep \psi} (T) \big) = \zeta_{cT} \big( \Upsilon(\varphi + \vep \psi) (T) \big)$. In view that $\Upsilon$ is of class $C^2$ around $\varphi$ we can make the expansion
\begin{equation}
\label{exp1}
u_{\varphi + \vep \psi} = \Upsilon (\varphi + \vep \psi) = \Upsilon(\varphi) + \vep \Upsilon'(\varphi) \psi + O(\vep^2).
\end{equation}
Now, from expression \eqref{uphi} we have that
\[
\begin{aligned}
u_{\varphi + \vep \psi}(t) &= \cV(t) (\varphi + \vep \psi) + \int_0^t \cV(t-\tau) F(u_{\varphi + \vep \psi}, \partial_x u_{\varphi + \vep \psi})(\tau) \, d\tau \\ 
&= \cV(t) (\varphi + \vep \psi) + \int_0^t \cV(t-\tau) \big[ g(u_{\varphi + \vep \psi}) - f'(u_{\varphi + \vep \psi}) \partial_x u_{\varphi + \vep \psi} \big] \, d\tau.
\end{aligned}
\]
Substituting \eqref{exp1} and recalling $\Upsilon(\varphi) = u_\varphi$ we arrive at 
\[
\begin{aligned}
g(u_{\varphi + \vep \psi}) &= g(u_\varphi) + \vep g'(u_\varphi) \Upsilon'(\varphi) \psi + O(\vep^2),\\
f'(u_{\varphi + \vep \psi}) \partial_x \Upsilon (\varphi + \vep \psi) &= f'(u_\varphi) \partial_x u_\varphi + \vep \partial_x \big( f'(u_\varphi) \Upsilon'(\varphi) \psi \big) + O(\vep^2).
\end{aligned}
\]
Upon substitution into the previous  integral formula one obtains
\[
\begin{aligned}
u_{\varphi + \vep \psi}(t) = \cV(t) \varphi + \int_0^t \cV(t-\tau) \big[ g(u_\varphi) - f'(u_\varphi) \partial_x u_\varphi \big] \, d\tau + \vep V_{\varphi,\psi}(t) +  O(\vep^2),
\end{aligned}
\]
with
\[
V_{\varphi,\psi}(t) :=
 \cV(t) \psi + \int_0^t \cV(t-\tau) \big[ g'(u_\varphi) \Upsilon'(\varphi) \psi - \partial_x \big( f'(u_\varphi) \Upsilon'(\varphi) \psi \big) \big] \, d\tau.
\]
Differentiation with respect to $\vep$ then yields
\[
\frac{d}{d\vep} \big( u_{\varphi + \vep \psi}(t) \big)_{|\vep = 0} = \frac{d}{d\vep} \big( \Upsilon(\varphi)(t) + \vep (\Upsilon'(\varphi) \psi)(t) + O(\vep^2) \big)_{|\vep = 0} = (\Upsilon'(\varphi) \psi)(t),
\]
and therefore
\[
V_{\varphi,\psi}(t) = (\Upsilon'(\varphi) \psi)(t) \in \Htper([0,L]),
\]
for all $t \in [0, {T}]$. Consequently, we have shown that $V_{\varphi,\psi}$  is a solution to the integral equation
\begin{equation}
\label{exprV}
V_{\varphi,\psi}(t) = \cV(t) \psi + \int_0^t \cV(t-\tau) \big[ g'(u_\varphi) V_{\varphi,\psi}(\tau) - \partial_x \big( f'(u_\varphi) V_{\varphi,\psi}(\tau) \big) \big] \, d\tau,
\end{equation}
for all $t \in [0, {T}]$. From formula \eqref{exprV} we recognize that $V_{\varphi,\psi}(0) = \psi$ and that it is the solution to the linearized Cauchy problem \eqref{Cplin} with $c=0$. We claim that
\begin{equation}
\label{claimV}
 v_\psi(t):=\zeta_{ct} \big(V_{\varphi,\psi}(t) \big),\quad  t \in [0, {T}],
\end{equation} 
is  the unique solution to the linearized Cauchy problem \eqref{Cplin} with initial datum $\psi$. Indeed, first notice that $\zeta_0 \big(V_{\varphi,\psi}(0) \big) = V_{\varphi,\psi}(0) = \psi$. Now, for $x \in [0,L]$ let us denote
\[
V(t,x) := \zeta_{ct} \big(V_{\varphi,\psi}(t) \big)(x) = V_{\varphi,\psi}(t)(x+ct) = V_{\varphi,\psi}(t, x + ct).
\]
Therefore, from expression \eqref{exprV} and in view that the viscous-dispersive operator $\cT = - \partial_x^3 + \partial_x^2$ is the infinitesimal generator of the semigroup $\cV(t)$, we obtain
\[
\begin{aligned}
\partial_t V &= \partial_t V_{\varphi,\psi}(t, x +ct) + c \partial_x V_{\varphi,\psi}(t,x+ct) \\
&= - \partial_x^3 V_{\varphi,\psi}(t,x+ct) + \partial_x^2 V_{\varphi,\psi}(t,x+ct) + g'(u_\varphi(x+ct)) V_{\varphi,\psi}(t,x+ct) +\\
&\quad - \partial_x \big( f'(u_{\varphi}(x+ct)) V_{\varphi,\psi}(t,x+ct) ) + c \partial_x V_{\varphi,\psi}(t,x+ct) \\
&= - \partial_x^3 V + \partial_x^2 V + g'(\varphi(x)) V - \partial_x \big( f'(\varphi(x)) V \big) + c \partial_x V,
\end{aligned}
\]
because $u_\varphi(\cdot + ct) = \varphi(\cdot -ct + ct) = \varphi(\cdot)$. This shows that $\partial_t V = \cL_0^c V$, $V(0) = \psi$ and therefore it is a solution to \eqref{Cplin}. By uniqueness of the solution, we obtain \eqref{claimV} for all $t \in [0, {T}]$.

Finally, evaluating at $t = T$ we have
\[
\begin{aligned}
\cS'(\varphi) \psi &= \frac{d}{d\vep} \Big( \zeta_{cT} (\Upsilon(\varphi) (T)) + \vep \zeta_{cT} ( \Upsilon'(\varphi) \psi (T) ) + O(\vep^2) \Big)_{|\vep = 0} \\
&= \zeta_{cT} \big( \Upsilon'(\varphi) \psi (T) \big)= \zeta_{cT} \big(V_{\varphi,\psi}(T) \big)= v_\psi(T),
\end{aligned}
\]
for any $\psi \in \Htper([0,L])$. This shows (c) and the lemma is proved.
\end{proof}

We are now able to prove the instability criterion.

\subsection{Proof of Theorem \ref{mainthem}}

From the hypotheses of the theorem, we know that the linearized operator around the $L$-periodic wave, $\cL^c \, : \, \Ldper([0,L]) \to \Ldper([0,L])$, has an unstable eigenvalue $\lambda \in \C$ with $\Re \lambda > 0$ and associated eigenfunction $\Psi \in \Htper([0,L])$. Hence, we now regard $\Psi$ as the initial condition for the linear Cauchy problem \eqref{Cplin}. From Theorem \ref{lemglobalwp} we know that there exists a unique global solution $v_\Psi \in C([0,T];\Htper([0,L])) \cap C^1([0,T];\Ldper([0,L])$ with $v_\Psi(0)=\Psi$, for any $T > 0$, fixed but arbitrary.

If we define, however, $U(t) = e^{\lambda t} \Psi \in \Htper([0,L])$ for all $t \geq 0$ then, clearly, 
\[
U \in C([0,T];\Htper([0,L])) \cap C^1([0,T]; \Ldper([0,L])),
\]
and $U(0) = \Psi$. Moreover, upon differentiation we obtain
\[
\partial_t U = \lambda e^{\lambda t} \Psi = e^{\lambda t} \cL_0^c \Psi = \cL_0^c \big( e^{\lambda t} \Psi \big) = \cL_0^c U.
\]
Hence, $U$ is also a solution to the Cauchy problem \eqref{Cplin} with $U(0) = \Psi$. By uniqueness of the solution we obtain $U(t) = v_\Psi (t)$ in $\Htper([0,L])$ for all $t \geq 0$. Now, define $\mu := e^{\lambda T}$. This implies that
\[
\cS'(\varphi) \Psi = v_\Psi  (T) = U(T) = e^{\lambda T} \Psi = \mu \Psi.
\]
This shows that $\mu \in \sigma(\cS'(\varphi))$ with $|\mu| > 1$ because $\Re \lambda > 0$. Thus, the mapping defined in \eqref{defS} on an open neighborhood of $\varphi$ satisfies the hypotheses of Corollary \ref{corhenry}. 
Therefore, for $\Gamma=\mathcal O_\varphi$ being a $C^1$-curve of fixed points of $\cS$, we conclude that the periodic traveling wave $\varphi$ is orbitally unstable in the space $\Htper([0,L])$. This finishes the proof.
%
%
%
%
\qed

\section{Application: the Korteweg-de Vries-Burgers-Fisher equation}
\label{secapplication}

Let us consider the following KdVB model with a source of logistic type,
\begin{equation}
\label{KdVBFn}
u_t + \alpha uu_x + u_{xxx} = u_{xx} + ru(1-u),
\end{equation}
for some fixed constants $\alpha, r > 0$. Equation \eqref{KdVBFn} was recently introduced by Ko\c{c}ak \cite{Koc20}, who proved the existence of traveling waves (such as solitons, kink and anti-kink waves, among others). The reaction term,
\[
f(u) = ru(1-u),
\]
is also known as a logistic reaction function or a source of Fisher-KPP type \cite{Fis37,KPP37}. It models the dynamics of populations with limited resources which saturate into a stable equilibrium point associated to an intrinsic carrying capacity. Consequently, \eqref{KdVBFn} is also called the \emph{Korteweg-de Vries-Burgers-Fisher (KdVBF) equation}.

In a recent contribution, Folino \emph{et al.} \cite{FNP24} proved both the existence and the spectral instability of a family small-amplitude periodic waves for the KdVBF equation \eqref{KdVBFn}. The existence theorem can be stated as follows.

\begin{theorem}[existence of small amplitude periodic waves \cite{FNP24}]
\label{theexistKdVBF} 
For any fixed parameter values $r,\alpha>0$, there exist $\epsilon_{0}>0$
sufficiently small and a critical speed value $c_{0}=-r$ such that the
KdV-Burgers-Fisher equation \eqref{KdVBFn} has a family of smooth periodic
traveling wave solutions of the form $u(x,t)=\varphi^{\epsilon}(x-c(\epsilon
)t)$, indexed by $0<\epsilon<\epsilon_{0}$, with fundamental period
\begin{equation}
L_{\epsilon}=\frac{2\pi}{\sqrt{r}}+O(\epsilon),\text{ as }\epsilon
\rightarrow0^{+}, \label{47}%
\end{equation}
and with amplitude growing like
\begin{equation}
\varphi^{\epsilon},(\varphi^{\epsilon})^{\prime}=O\left(  \sqrt{\epsilon
}\right)  . \label{48}%
\end{equation}
The speed of propagation is given by a function $c(\epsilon)=-r+O(\epsilon),$
$\epsilon\in\left(  0,\epsilon_{0}\right)  ,$ such that $c(\epsilon
)\rightarrow c_{0}$ as $\epsilon\rightarrow0^{+}$. For each fixed $\epsilon
\in(0,\epsilon_{0})$ the periodic wavetrain $\varphi^{\epsilon}$ is unique up
to translations.
\end{theorem}
\begin{proof}
See the proof of Theorem 2.1, \S 2.1, in Folino \emph{et al.} \cite{FNP24} for details.
\end{proof}

\begin{remark}
The proof of this existence result is based on a local bifurcation analysis. It shown that for any fixed positive values of the physical parameters $\alpha > 0$ and $r > 0$, a family of small-amplitude and finite period waves emerges from a subcritical Hopf bifurcation around a critical value of the wave speed given by $c_0 = -r$.
\end{remark}

Formulae (\ref{47}) and (\ref{48}) imply that, for a
fixed small $\epsilon$, the fundamental period of the wave is of order $O(1)$
and the amplitude of the waves is of order $O(\sqrt{\epsilon})$, respectively.
Thus, one expects that when $\epsilon\rightarrow0^{+}$ the small amplitude
waves tend to the origin and the linearized operator formally becomes a
constant coefficient linearized operator around the zero solution. This
observation is the basis of the analysis by Folino \emph{et al.} \cite{FNP24}, who proved that
unstable point eigenvalues of this constant coefficient operator split into
neighboring curves of Floquet spectra of the underlying small amplitude waves. Indeed, the Bloch family of linearized operators around the wave (cf. Kapitula and Promislow \cite{KaPro13} and Folino \emph{et al.} \cite{FNP24}),
\begin{equation}
\label{Blochoper}
\left\{
\begin{aligned}
\mathcal{L}_{\theta}^{c(\epsilon)}  &  :=-(\partial_{x}+i\theta/L_{\epsilon
})^{3}+(\partial_{x}+i\theta/L_{\epsilon})^{2}-a_{1}^{\epsilon}(x)(\partial
_{x}+i\theta/L_{\epsilon})-a_{0}^{\epsilon}(x) \Id,\\
\mathcal{L}_{\theta}^{c(\epsilon)}  &  :\Ldper([0,L_{\epsilon
}])\rightarrow \Ldper([0,L_{\epsilon}]), 
\end{aligned}
\right.
\end{equation}
where
\[
\begin{aligned}
a_{1}^{\epsilon}(x)&:=-c(\epsilon)-\alpha\varphi^{\epsilon},\\
a_{0}^{\epsilon}(x)&:=r-2r\varphi^{\epsilon}-\alpha(\varphi^{\epsilon})_{x},
\end{aligned}
\]
for each $\theta\in(-\pi,\pi]$, with domain $D(\mathcal{L}_{\theta
}^{c(\epsilon)})=\Htper([0,L_{\epsilon}])$, can be transformed into a
family of operators, $\widetilde{\mathcal{L}}_{\theta}^{\epsilon}$ , defined
on the periodic space $\Ldper([0,\pi])$ for which the period no longer
depends on $\epsilon$.

For that purpose, the authors in Folino \emph{et al.} \cite{FNP24} make the change of variables,
$y:=\pi x/L_{\epsilon}$ and $w(y):=u(L_{\epsilon}y/\pi)$, and apply (\ref{47})
and $c(\epsilon)=c_{0}+O(\epsilon)$, in order to recast the spectral problem
for the operators in \eqref{Blochoper} as $\widetilde{\mathcal{L}}_{\theta}%
^{\epsilon}w=\lambda w$, where
\begin{align*}
\widetilde{\mathcal{L}}_{\theta}^{\epsilon}  &  :=\widetilde{\mathcal{L}%
}_{\theta}^{0}+\sqrt{\epsilon}\widetilde{\mathcal{L}}_{\theta}^{1}:\Ldper([0,\pi])\rightarrow \Ldper([0,\pi]),\\
\widetilde{\mathcal{L}}_{\theta}^{0}  &  :=-(i\theta+\pi\partial_{y}%
)^{3}+L_{0}(i\theta+\pi\partial_{y})^{2}-L_{0}^{2}c_{0}\left(  i\theta
+\pi\partial_{y}\right)  +rL_{0}^{3}\Id,\\
\widetilde{\mathcal{L}}_{\theta}^{1}  &  :=b_{2}\left(  i\theta+\pi
\partial_{y}\right)  ^{2}-b_{1}(y)(i\theta+\pi\partial_{y})-b_{0}(y)\Id,
\end{align*}
for each $\theta\in(0,\pi]$ and where the coefficients behave like
\begin{align*}
b_{0}(y)  &  :=\frac{1}{\sqrt{\epsilon}}\left(  a_{0}^{\epsilon}\left(
y\right)  +rL_{0}^{3}\right)  =O\left(  1\right)  ,\\
b_{1}(y)  &  :=\frac{1}{\sqrt{\epsilon}}\left(  a_{1}^{\epsilon}\left(
y\right)  -L_{0}^{2}c_{0}\right)  =O\left(  1\right)  ,\\
b_{2}  &  :=\frac{L_{\epsilon}-L_{0}}{\sqrt{\epsilon}}=O(1)
\end{align*}
as $\epsilon\rightarrow 0^+$. It can be shown that,
for every $\theta$, $\widetilde{\mathcal{L}}_{\theta}^{1}$ is $\widetilde
{\mathcal{L}}_{\theta}^{0}$ -bounded (see Lemma 4.3 in Folino \emph{et al.} \cite{FNP24}). Therefore,
upon application of standard perturbation theory for linear operators (cf.
Kato \cite{Kat80}), it is shown that both spectra, $\sigma (\widetilde
{\mathcal{L}}_{\theta}^{\epsilon})  $ and $\sigma(  \widetilde
{\mathcal{L}}_{\theta}^{0})  ,$ are located nearby in the complex plane
for $\epsilon>0$ small enough.

Transforming back into the original coordinates, the same conclusion holds for
any fixed, sufficiently small $\epsilon>0$ and the associated family of Bloch
operators \eqref{Blochoper} defined on $\Ldper\left(  [0,L_{\epsilon}]\right)
$. In particular, for $\theta=0$, the unperturbed operator
\begin{equation}
\begin{aligned}
\mathcal{L}_{0}^{c(0)}  &  :=-\partial_{x}^{3}+\partial_{x}^{2} - r\partial
_{x}+r\Id,\\
\mathcal{L}_{0}^{c(0)}  &  :\Ldper([0,L_{0}])\rightarrow \Ldper([0,L_{0}]),
\end{aligned}
\end{equation}
with domain $D(\mathcal{L}_{0}^{c(0)})=\Htper([0,L_{0}])$, has
a positive eigenvalue $\widetilde{\lambda}_{0}=r$ associated to the constant
eigenfunction $\Psi_{0}(y)=1\in \Htper([0,L_{0}])$. Hence, the operator
$\mathcal{L}_{0}^{c(\epsilon)}$ has discrete eigenvalues $\widetilde
{\lambda_{j}}\left(  \epsilon\right)  $ in a $\sqrt{\epsilon}$-neighborhood of
$\widetilde{\lambda}_{0}=r$ with multiplicities adding up to the multiplicity
of $\widetilde{\lambda}_{0}$ provided that $\epsilon$ is sufficiently small.
Moreover, since $\widetilde{\lambda}_{0}>0$ there holds $\operatorname{Re}%
\widetilde{\lambda_{j}}\left(  \epsilon\right)  >0$. Henceforth, we have the
following result.

\begin{lemma}
\label{Lemma 8} For each $0<\epsilon\ll1$ sufficiently small there holds
\begin{equation}
\ptsp (  \mathcal{L}_{0}^{c(\epsilon)})_{| \Ldper}  \cap\{\lambda\in\mathbb{C}:|\lambda-rL_{0}^{3}%
|<\zeta(\epsilon)\}\neq\varnothing, \label{50}%
\end{equation}
for some $\zeta(\epsilon)=O(\epsilon)>0.$
\end{lemma}

\begin{proof}
See Lemma 4.8 and the proof of Theorem 4.1 in Folino \emph{et al.} \cite{FNP24}.
\end{proof} 

Therefore, for each $0 < \epsilon \ll 1$ sufficiently small we conclude the existence of an unstable eigenvalue $\lambda
(\epsilon)\in\mathbb{C}$ with $\operatorname{Re}\lambda(\epsilon)>0$ and an
eigenfunction $\Psi^{\epsilon}\in \Htper([0,L_{\epsilon}])$, such
that $\mathcal{L}_{0}^{c(\epsilon)}\Psi^{\epsilon}=\lambda(\epsilon
)\Psi^{\epsilon}$, that is, the spectral instability property holds. Hence,
upon application of Theorem \ref{mainthem}, we have the following result.

\begin{theorem}[orbital instability of small-amplitude periodic waves]
\label{Theorem 6} 
There exists $\overline{\epsilon}_{0}\in\left(  0,\epsilon_{0}\right)  $
sufficiently small such that each periodic wave from Theorem \ref{theexistKdVBF},
$u(x,t)$ $=\varphi^{\epsilon}(x-c(\epsilon)t)$, with $\epsilon\in\left(
0,\overline{\epsilon}_{0}\right)  $, is orbitally unstable in the periodic
space $\Htper([0,L_{\epsilon}])$ under the flow of the KdVBF equation \eqref{KdVBFn}.
\end{theorem}

\begin{remark}
It is to be observed that the instability of small amplitude waves is a consequence of the instability of the origin as an equilibrium point of the reaction term: heuristically, the linearized operator around the wave tends to a spectrally unstable constant coefficient operator when the amplitude of the wave tends to zero.
\end{remark}

\section*{Acknowledgements}

A. Naumkina was supported by CONAHCyT (Mexico), through a scholarship for graduate studies (CVU 1081529). The work of R. G. Plaza was partially supported by CONAHCyT, Program ``Ciencia de Frontera'', project CF-2023-G-122.

\def\cprime{$'\!\!$} \def\cprimel{$'\!$}





\end{document}